\documentclass[final]{siamltex}

\usepackage{xspace}
\usepackage{multirow}
\usepackage{array}
\usepackage{ifthen}
\usepackage{amsmath}
\usepackage{amsfonts}
\usepackage{color}
\usepackage{verbatim}
\usepackage{relsize}
\usepackage{mathdots}
\usepackage{array}
\usepackage{rotating}
\usepackage{hyperref}

\definecolor{grey}{gray}{0.6}

\usepackage{amssymb}

\usepackage[mathscr]{eucal}

\usepackage{stmaryrd}

\usepackage{amsbsy}

\usepackage{bm}

\usepackage{braket}

\newcommand{\V}[2][]{{\bm{#1 \mathbf{\MakeLowercase{#2}}}}} 
\newcommand{\M}[2][]{{\bm{#1 \mathbf{\MakeUppercase{#2}}}}} 
\newcommand{\T}[2][]{\boldsymbol{#1 \mathscr{\MakeUppercase{#2}}}}

\newcolumntype{I}{!{\vrule width 1pt}}

\newlength\savedwidth
\newcommand\whline{\noalign{\global\savedwidth\arrayrulewidth
                            \global\arrayrulewidth 1pt}%
           \hline
           \noalign{\global\arrayrulewidth\savedwidth}}

\newlength{\arrayrulewidthOriginal}

\newcommand{\matsz}[2]{\ensuremath{\mathbb{R}^{#1 \times #2}}}
\newcommand{\tensz}[2]{\ensuremath{\mathbb{R}^{[#1,#2]}}}

\newcommand{\becomes}{:=}

\newcommand{\BCSS}{Blocked Compact Symmetric Storage\xspace}
\newcommand{\BCSSshort}{BCSS\xspace}
\newcommand{\Dense}{Dense\xspace}

\newcommand{\Vrow}[1]{{\widehat {\bf{\MakeLowercase{#1}}}}^T}
\newcommand{\VrowT}[1]{{\widehat {\bf{\MakeLowercase{#1}}}}}

\newcommand{\blkDim}[1]{b_{#1}}
\newcommand{\blkedDimC}{\bar p}
\newcommand{\blkedDimA}{\bar n}
\newcommand{\bA}{b_{\M{A}}}
\newcommand{\bC}{b_{\M{C}}}
\newcommand{\nB}{\bar n}
\newcommand{\pB}{\bar p}

\newcommand{\blkI}{\bar \imath}
\newcommand{\blkJ}{\bar \jmath}

\newcommand{\IdxFirst}{0}
\newcommand{\IdxSecond}{1}
\newcommand{\IdxThird}{2}
\newcommand{\IdxFourth}{3}
\newcommand{\IdxNext}[1]{#1}
\newcommand{\IdxSecondNext}[1]{#1+1}
\newcommand{\IdxThirdNext}[1]{#1+2}
\newcommand{\IdxFourthNext}[1]{#1+3}
\newcommand{\IdxLast}[1]{#1-1}
\newcommand{\IdxSecondLast}[1]{#1-2}
\newcommand{\subtwo}[2]{#1,#2}

\newcommand{\IdxParenLast}[1]{(#1-1)}

\renewcommand{\subtwo}[2]{#1#2}
\newcommand{\subthree}[3]{\subtwo{#1}{\subtwo{#2}{#3}}}
\newcommand{\subfour}[4]{\subtwo{#1}{\subthree{#2}{#3}{#4}}}

\newcommand{\fromtoMartin}[2]{#2}

\newcommand{\Eqn}[1]{(\ref{eq:#1})}
\newcommand{\Fig}[1]{Figure~\ref{fig:#1}}
\newcommand{\Tbl}[1]{Table~\ref{tbl:#1}}

\newcommand{\Thm}[1]{Theorem~\ref{thm:#1}}
\newcommand{\R}{\mathbb{R}}

\everymath{\displaystyle}

\title{Exploiting Symmetry in Tensors for High Performance:\\
Multiplication with Symmetric Tensors}

\author{
Martin D. Schatz\footnotemark[2]
 \and
Tze Meng Low\footnotemark[2] \and
Robert A. van de Geijn\footnotemark[2]
\and
Tamara G. Kolda\footnotemark[4]
}

\date{Draft \\ \today}
\begin{document}

\maketitle

\renewcommand{\thefootnote}{\fnsymbol{footnote}}
\footnotetext[2]{
Department of Computer Science, Institute for Computational Engineering and Sciences, The University of Texas at Austin,
Austin, TX.\newline Emails: martin.schatz@utexas.edu, ltm@cs.utexas.edu, rvdg@cs.utexas.edu.
}
\footnotetext[4]{Sandia National Laboratories,
Livermore, CA. Email: tgkolda@sandia.gov.}

\begin{abstract}
Symmetric tensor operations arise in a wide variety of computations.
However, the benefits of exploiting symmetry in order to reduce storage and
computation is in conflict with a desire to simplify memory access patterns.  In
this paper, we propose a blocked data structure (\BCSS) wherein we consider the
tensor by blocks and store only the unique blocks of a symmetric tensor.
We propose an algorithm-by-blocks, already shown of benefit for matrix
computations, that exploits this storage format by utilizing a series of
temporary tensors to avoid redundant computation.  Further, partial symmetry
within temporaries is exploited to further avoid redundant storage and redundant
computation.
A detailed analysis shows that, relative to storing and computing with tensors
without taking advantage of symmetry and partial symmetry, storage requirements
are reduced by a factor of $ O\left( m! \right)$ and computational requirements
by a factor of $O\left( (m+1)!/2^m \right)$, where $ m $ is the order of the
tensor.  \fromtoMartin{}{However, as the analysis shows, care must be taken in
choosing the correct block size to ensure these storage and computational
benefits are achieved (particularly for low-order tensors).}
 An implementation demonstrates that storage is greatly reduced and the
 complexity introduced by storing and computing with tensors by blocks is
 manageable. Preliminary results demonstrate
that computational time is also reduced.
The paper concludes with a discussion of how insights in this paper point to opportunities
for generalizing recent advances in the domain of linear algebra libraries to
the field of multi-linear computation.

%
%
%
%
\end{abstract}

\newtheorem{tendef}{Definition}

\section{Introduction}
A tensor is a multi-dimensional or $m$-way array. Tensor computations are
increasingly prevalent in a wide variety of
applications~\cite{Kolda:2009:SIAM:tensor-review}.  Alas, libraries for dense
multi-linear algebra (tensor computations) are in their infancy. 
\fromtoMartin{The target objects of this paper are symmetric tensors; tensors
whose entries are invariant under any permutation of indices.  The aim of this
paper is to explore how ideas from matrix computations can be extended to the
domain of tensors.  Specifically, this paper focuses on
exploring how exploiting symmetry in matrix computations extends to computations
with symmetric tensors and exploring how block structures and algorithms extend
to computations with symmetric tensors.}{  The aim of this
paper is to explore how ideas from matrix computations can be extended to the
domain of tensors.  Specifically, this paper focuses on
exploring how exploiting symmetry in matrix computations extends to computations
with symmetric tensors, tensors whose entries are invariant under any
permutation of indices, and exploring how block structures and algorithms extend
to computations with symmetric tensors.}
 
Libraries for dense linear algebra (matrix computations) have long been part of
the standard arsenal for computational science, including the BLAS
interface~\cite{BLAS1,BLAS2,BLAS3,Goto:2008:AHP,1377607}, LAPACK~\cite{LAPACK3},
and more recent libraries with similar functionality, like the BLAS-like
Interface Software framework (BLIS)~\cite{FLAWN66}, and {\tt
libflame}~\cite{libflame_ref,CiSE09}.  For distributed memory architectures, the
ScaLAPACK~\cite{ScaLAPACK}, PLAPACK~\cite{PLAPACK}, and
Elemental~\cite{Poulson:2012:ENF} libraries provide most of the functionality of
the BLAS and LAPACK.  High-performance implementations of these libraries are
available under open source licenses.

For tensor computations, no high-performance general-purpose libraries exist.
The MATLAB Tensor Toolbox~\cite{TTB_Software,TTB_Dense} defines many commonly
used operations that would be needed by a library for multilinear algebra but
does not have any high-performance kernels nor special computations or data
structures for symmetric tensors.   The PLS Toolbox~\cite{PLS_Eigenvector} provides users
with operations for analyzing data stored as tensors, but does not expose the
underlying system for users to develop their own set of operations. Targetting
distributed-memory environments, the Tensor Contraction Engine (TCE)
project~\cite{TCE} focuses on sequences of tensor contractions and uses compiler
techniques to reduce workspace and operation counts.  The Cyclops Tensor
Framework (CTF)~\cite{CTF} focuses on exploiting symmetry in storage for
distributed memory parallel computation with tensors, but at present does not
include efforts to optimize computation within each computational node.

In a talk at the Eighteenth Householder Symposium meeting
(2011), Charlie Van Loan stated, ``In my opinion, blocking will eventually
have the same impact in tensor computations as it does in matrix computations.''
The approach we take in this paper heavily borrows from the FLAME
project~\cite{CiSE09}.  We use the {\em change-of-basis} operation, also known
as a \underline{S}ymmetric \underline{T}ensor \underline{T}imes \underline{S}ame
\underline{M}atrix (in all modes) ({\tt sttsm}) operations~\cite{TTB_Software},
to motivate the issues and solutions.  In the field of computational chemistry,
this operation is referred to as an atomic integral
transformation~\cite{Bender1972547} when applied to order-$4$ tensors. This
operation appears in other contexts as well, such as computing a low-rank
Tucker-type decomposition of a symmetric tensor~\cite{VanDooren:Jacobi}
\fromtoMartin{}{and blind source separation~\cite{Regalia2013875}}. We propose
algorithms that require significantly less (possibly minimal) computation
\fromtoMartin{}{relative to an approach based on a series of successive
matrix-matrix multiply operations} by computing and storing temporaries.
Additionally, the tensors are stored by blocks, following similar solutions
developed for matrices~\cite{FLAWN12,SuperMatrix:TOMS}.  In addition to many of
the projects mentioned previously, other work, such as that by Van Loan and
Ragnarsson~\cite{Ragnarsson2013853} suggest devising algorithms in terms of
tensor blocks to aid in computation with both symmetric tensors and tensors in
general.

Given that we store the tensor by blocks, the algorithms must be reformulated to
operate with these blocks.
Since we need only store the unique blocks of a symmetric tensor, symmetry is
exploited at the block level (both for storage and computation) while preserving
regularity when computing within blocks. Temporaries are formed to reduce the
computational requirements, similar to work in the field of computational
chemistry~\cite{Bender1972547}.  To further reduce computational and storage
requirements, we exploit partial symmetry within temporaries.  It should be
noted that the symmetry being exploited in this article is different from the
symmetry typically observed in chemistry fields.  One approach for exploiting
symmetry in operations is to store only unique entries and devise algorithms
which only use the unique entries of the symmetric
operands~\cite{Yamamoto200558}.  By contrast, we exploit symmetry in operands by
devising algorithms and storing the objects in such a way that knowledge of the
symmetry of the operands is concealed from the implementation (allowing
symmetric objects to be treated as non-symmetric objects).

The contributions of this paper can be summarized as reducing storage and
computational requirements of the {\tt sttsm} operation for symmetric tensors
by:

\begin{itemize}
  \item Utilizing temporaries to reduce computational costs thereby avoiding redundant computations.
  \item Using blocked algorithms and data structures to improve performance of the given computing environment. 
  \item Providing a framework for exploiting symmetry in symmetric tensors (and partial symmetry in temporaries) thereby reducing storage requirements.
\end{itemize}

The paper analyzes the computational and storage costs demonstrating that the
added complexity of exploiting symmetry need not adversely impact the benefits
derived from symmetry.  An implementation shows that the insights can be made
practical.  The paper concludes by listing additional opportunities for
generalizing advancements in the domain of linear algebra libraries to multi-linear
computation.

%
\section{Preliminaries}

We start with some basic notation, and the motivating tensor operation.

\subsection{Notation}

In this discussion, we assume all indices and modes are numbered
starting at zero.

The \emph{order} of a tensor is the number of ways or modes. In this paper, we
deal only with tensors where every mode has the same dimension. Therefore, we
define $\tensz{m}{n}$ to be the set of real-valued order-$m$ (or $m$-way) tensors where each
mode has dimension $n$; i.e., a tensor $\T{A} \in \tensz{m}{n}$ can be thought
of as an $m$-dimensional cube with $n$ entries in each direction. 

An element of $\T{A}$ is denoted as $\alpha_{i_{\IdxFirst}\cdots
  i_{\IdxLast{m}}}$ where $i_k \in \set{\IdxFirst, \dots, \IdxLast{n}}$ for
all $k \in \set{\IdxFirst, \dots, \IdxLast{m}}$.
This also illustrates that, as a general rule, we use lower case Greek letters
for scalars ($ \alpha, \chi, \ldots $), bold lower case Roman letters for
vectors ($ \V{A}, \V{X}, \ldots $), bold upper case Roman letters for matrices
($ \M{A}, \M{X}, \ldots $), and upper case scripted letters for tensors ($
\T{A}, \T{X}, \ldots $). We denote the $i$th row of a matrix $
\M{A} $ by $ \Vrow{a}_i $.  If we transpose this row, we denote it as
$\VrowT{a}_i$.  

\subsection{Partitioning}
For our forthcoming discussions, it is useful to define the notion of a
\emph{partitioning} of a set $\mathcal{S}$.  We say the sets
$\mathcal{S}_\IdxFirst,\mathcal{S}_\IdxSecond,\ldots,\mathcal{S}_{\IdxLast{k}}$
form a \emph{partitioning} of $\mathcal{S}$ if
\[ 
\mathcal{S}_i \cap \mathcal{S}_j  = \emptyset \text{ for any } i,j \in
\{\IdxFirst,\dots,\IdxLast{k}\} \text{ with } {i \neq j}, 
\] 
\[ 
\mathcal{S}_i
\neq \emptyset \text{ for any } i \in \{\IdxFirst,\dots,\IdxLast{k}\}, 
\] 
and
\[
\bigcup_{i=\IdxFirst}^{\IdxLast{k}} \mathcal{S}_i = \mathcal{S}.
\]

\subsection{Partial Symmetry}
It is possible that a tensor $\T{A}$ may be symmetric in 2 or more modes,
meaning that the entries are invariant to permutations of those modes. For
instance, if $\T{A}$ is a 3-way tensor and symmetric in all modes, then 
\[
\alpha_{\subthree{i_{\IdxFirst}}{i_{\IdxSecond}}{i_{\IdxThird}}} = 
\alpha_{\subthree{i_{\IdxFirst}}{i_{\IdxThird}}{i_{\IdxSecond}}} = 
\alpha_{\subthree{i_{\IdxSecond}}{i_{\IdxFirst}}{i_{\IdxThird}}} = 
\alpha_{\subthree{i_{\IdxSecond}}{i_{\IdxThird}}{i_{\IdxFirst}}} = 
\alpha_{\subthree{i_{\IdxThird}}{i_{\IdxFirst}}{i_{\IdxSecond}}} =
\alpha_{\subthree{i_{\IdxThird}}{i_{\IdxSecond}}{i_{\IdxFirst}}}.
\] 
It may also be that $\T{A}$ is only symmetric in a subset of the modes. For
instance, suppose $\T{A}$ is a 4-way tensor that is symmetric in modes
$\mathcal{S} = \{\IdxSecond,\IdxThird\}$. Then 
\[ 
\alpha_{\subfour{i_{\IdxFirst}}{i_{\IdxSecond}}{i_{\IdxThird}}{i_{\IdxFourth}}} = \alpha_{\subfour{i_{\IdxFirst}}{i_{\IdxThird}}{i_{\IdxSecond}}{i_{\IdxFourth}}}.
\] 
We define this formally below.

Let $\mathcal{S}$ be a finite set. Define $\Pi_{\mathcal{S}}$ to be the set of
all permutations on the set $\mathcal{S}$ where a permutation is viewed as a
bijection from $\mathcal{S}$ to $\mathcal{S}$. Under this interpretation, for
any $\pi \in \Pi_{\mathcal{S}}$, $\pi(x)$ is the resulting element of applying
$\pi$ to $x$\footnote{Throughout this paper, $\pi$ should be interpreted as a
permutation, not as a scalar quantity. All other lowercase Greek letters should
be interpreted as scalar quantities.}.

Let $\mathcal{S} \subseteq \{\IdxFirst,\dots,\IdxLast{m}\}$, and define $\Pi_{S}$ to be the set
of all permutations on $\mathcal{S}$ as described above.
We say an order-$m$ tensor $\T{A}$ is \emph{symmetric in the modes in
$\mathcal{S}$} if
\[ 
\alpha_{\subfour{i'_{\IdxFirst}}{i'_{\IdxSecond}}{\cdots}{i'_{\IdxLast{m}}}} 
=
\alpha_{\subfour{i_{\IdxFirst}}{i_{\IdxSecond}}{\cdots}{i_{\IdxLast{m}}}}
\] 
for any index vector $i'$ defined by
\[
i'_j = \begin{cases} 
\pi(i_j) & \text{if } j \in \mathcal{S}, \\
i_j & \text{otherwise}
\end{cases}
\] 
for $j = \IdxFirst,\dots,\IdxLast{m}$ and $\pi \in \Pi_{\mathcal{S}}$.

Technically speaking, this definition applies even in the trivial case where
$\mathcal{S}$ is a singleton, which is useful for defining multiple symmetries.

\subsection{Multiple Symmetries}
\label{sec:mult_sym}
It is possible that a tensor may have symmetry in multiple sets of modes at
once. As the tensor is not symmetric in all modes, yet still symmetric in some
modes, we say the tensor is \emph{partially}-symmetric.  For instance, suppose
$\T{A}$ is a 4-way tensor that is symmetric in modes $\mathcal{S}_\IdxFirst =
\{\IdxSecond,\IdxThird\}$ and also in modes $\mathcal{S}_\IdxSecond =
\{\IdxFirst,\IdxFourth\}$. Then

\[
\alpha_{\subfour{i_{\IdxFirst}}{i_{\IdxSecond}}{i_{\IdxThird}}{i_{\IdxFourth}}} =
\alpha_{\subfour{i_{\IdxFourth}}{i_{\IdxSecond}}{i_{\IdxThird}}{i_{\IdxFirst}}} =
\alpha_{\subfour{i_{\IdxFirst}}{i_{\IdxThird}}{i_{\IdxSecond}}{i_{\IdxFourth}}} =
\alpha_{\subfour{i_{\IdxFourth}}{i_{\IdxThird}}{i_{\IdxSecond}}{i_{\IdxFirst}}}.
\] 
We define this formally below.

Let $\mathcal{S}_{\IdxFirst},\mathcal{S}_{\IdxSecond},\ldots,\mathcal{S}_{\IdxLast{k}}$ be a partitioning of
$\{\IdxFirst,\dots,\IdxLast{m}\}$.
We say an order-$m$ tensor $\T{A}$ has symmetries defined by the mode
partitioning $\{S_i\}_{i=\IdxFirst}^{\IdxLast{k}}$ if

\[ 
\alpha_{\subfour{i'_{\IdxFirst}}{i'_{\IdxSecond}}{\cdots}{i'_{\IdxLast{m}}}} 
=
\alpha_{\subfour{i_{\IdxFirst}}{i_{\IdxSecond}}{\cdots}{i_{\IdxLast{m}}}}
\] 
for any index vector $i'$ defined by
\[
i'_j = \begin{cases} 
\pi_\IdxFirst(i_j) & \text{if } j \in \mathcal{S}_\IdxFirst, \\
\pi_\IdxSecond(i_j) & \text{if } j \in \mathcal{S}_\IdxSecond, \\
\vdots \\
\pi_{\IdxLast{k}}(i_j) & \text{if } j \in \mathcal{S}_{\IdxLast{k}} \\
\end{cases}
\] 
for $j = \IdxFirst,\dots,\IdxLast{m} \text{ and } \pi_{\ell} \in \Pi_{\mathcal{S}_{\ell}}$ for $\ell = \IdxFirst,\dots,\IdxLast{k}$.

Technically, a tensor with no symmetry whatsoever still fits the definition
above with $k=m$ and $\left\vert\mathcal{S}_i\right\vert = 1$ for
$i=\IdxFirst,\dots,\IdxLast{m}$.
If $k=1$ and $\mathcal{S}_\IdxFirst = \{\IdxFirst,\dots,\IdxLast{m}\}$, then the tensor is \emph{symmetric}.
If $1 < k < m$, then the tensor is \emph{partially symmetric}.
Later, we look at partially symmetric tensors such that
$\mathcal{S}_\IdxFirst = \{\IdxFirst,\dots,\ell\}$ and $\left\vert\mathcal{S}_i\right\vert = 1$ for $i=\IdxSecond,\ldots,\IdxLast{k}$.

\subsection{The {\tt sttsm} operation}

The operation used in this paper to illustrate issues related to
storage of, and computation with, symmetric tensors 
is the \emph{change-of-basis} operation 
\begin{equation}
  \label{eq:cob}
  \T{C} \becomes [\T{A}; \underbrace{\M{X}, \cdots, \M{X}}_{\text{$m $ times}}] = \T{A} \times_{\IdxFirst} \M{X} \times_{\IdxSecond} \cdots \times_{\IdxLast{m}} \M{X},
\end{equation} 
where $ \T{A} \in \tensz{m}{n} $ is symmetric and $ \M{X} \in
\matsz{p}{n}$ is the change-of-basis matrix.   
This is equivalent to multiplying the tensor $\T{A}$ by the same
matrix $\M{X}$ in every mode.
The resulting tensor
$\T{C} \in \tensz{m}{p}$ is defined elementwise as
\begin{displaymath}
\gamma_{\subthree{j_\IdxFirst}{\cdots}{j_{\IdxLast{m}}}} \becomes 
\sum_{i_\IdxFirst=\IdxFirst}^{\IdxLast{n}} \cdots \sum_{i_{\IdxLast{m}}=\IdxFirst}^{\IdxLast{n}} 
\alpha_{\subthree{i_\IdxFirst}{\cdots}{i_{\IdxLast{m}}}} 
\chi_{\subtwo{j_\IdxFirst}{i_\IdxFirst}}
\chi_{\subtwo{j_\IdxSecond}{i_\IdxSecond}} 
\cdots \chi_{\subtwo{j_{\IdxLast{m}}}{i_{\IdxLast{m}}}},  
\end{displaymath}
where $j_k \in \{\IdxFirst,\ldots,\IdxLast{p}\}$ for all $k \in \set{\IdxFirst,\dots,\IdxLast{m}}$.
It can be observed that the resulting tensor $\T{C}$ is itself symmetric.
We refer to this operation \Eqn{cob} as the {\tt{sttsm}} operation.

The {\tt sttsm} operation is used in computing symmetric versions of Tucker and
CP (notably the CP-opt) decompositions for symmetric
tensors~\cite{Kolda:2009:SIAM:tensor-review}.  In the CP decomposition, the
matrix $\M{X}$ of the {\tt sttsm} operation is a single vector.  In the field of
computational chemistry, the {\tt sttsm} operation is used when transforming
atomic integrals~\cite{Bender1972547}.  Many fields utilize a closely-related
operation to the {\tt sttsm} operation, which can be viewed as the
multiplication of a symmetric tensor in all modes {\bf but one}.  Problems such
as calculating Nash-Equlibria for symmetric games~\cite{Cheng04noteson} utilize
this related operation.  We focus on the {\tt sttsm} operation not only to
improve methods that rely on this exact operation, but also to gain insight for
tackling related problems of symmetry in related operations.

%
%

%
%
%
%
%
%
%
%
%
%

%
%
%
%

\section{The Matrix Case}
\label{sec:MatrixCase}

We build intuition about the problem and its solutions by first looking at
symmetric matrices (order-2 symmetric tensors).

\subsection{The operation for $m=2$}

Letting $m = 2$ yields 
$\M{C} \becomes [ \M{A}; \M{X}, \M{X} ]$
where $\M{A} \in \tensz{m}{n}$ 
is an $n \times n$ symmetric matrix,
$\M{C} \in \tensz{m}{p}$ is a $p \times p$ symmetric matrix,
and $[ \M{A}; \M{X}, \M{X} ] = \M{X} \M{A} \M{X}^T $.
For $m=2$, \Eqn{cob} becomes
\begin{alignat}{2}
\label{eqn:matrix1}
\gamma_{\subtwo{j_{\IdxFirst}}{j_{\IdxSecond}}} = 
\sum_{i_{\IdxFirst}=\IdxFirst}^{\IdxLast{n}}
\sum_{i_{\IdxSecond}=\IdxFirst}^{\IdxLast{n}}
\alpha_{\subtwo{i_{\IdxFirst}}{i_{\IdxSecond}}} 
\chi_{\subtwo{j_{\IdxFirst}}{i_{\IdxFirst}}} \chi_{\subtwo{j_{\IdxSecond}}{i_{\IdxSecond}}}.
\end{alignat}

\subsection{Simple algorithms for $m=2$}

\begin{figure}[p]
\begin{center}
\begin{tabular}{| p{2.5in} | p{2.25in} |} \hline
Naive algorithms
&
Algorithms that reduce computation at the expense
  of extra workspace
\\ \hline 
\hline
\multicolumn{2}{|l|}{
$ \T{A} $ is a matrix ($ m = 2 $): $ \M{C} \becomes \M{X} \M{A}
\M{X}^T = [\M{A};\M{X},\M{X}] $}
\\ \hline
\begin{minipage}{2.5in}
\begin{tabular}[t]{@{}l}
{\bf for} $ j_{\IdxSecond} = \IdxFirst, \ldots, \IdxLast{p} $ \\
~~~{\bf for} $ j_{\IdxFirst} = \IdxFirst, \ldots, j_{\IdxSecond} $ \\
~~~~~~$ \gamma_{\subtwo{j_{\IdxFirst}}{j_{\IdxSecond}}} \becomes 0 $ \\
~~~~~~{\bf for} $ i_{\IdxFirst} = \IdxFirst, \ldots, \IdxLast{n} $ \\
~~~~~~~~~{\bf for} $ i_{\IdxSecond} = \IdxFirst, \ldots, \IdxLast{n} $ \\
~~~~~~~~~~~~$ \gamma_{\subtwo{j_{\IdxFirst}}{j_{\IdxSecond}}} \becomes \gamma_{\subtwo{j_{\IdxFirst}}{j_{\IdxSecond}}} + \alpha_{\subtwo{i_{\IdxFirst}}{i_{\IdxSecond}}} \chi_{\subtwo{j_{\IdxFirst}}{i_{\IdxFirst}}}
 \chi_{\subtwo{j_{\IdxSecond}}{i_{\IdxSecond}}} $ \\
~~~~~~~~~{\bf endfor} \\
~~~~~~{\bf endfor} \\
~~~{\bf endfor} \\
{\bf endfor}
\end{tabular}
\end{minipage}
&
\begin{minipage}[c]{1.5in}
\begin{tabular}{@{}l}
{\bf for} $ j_{\IdxSecond} = \IdxFirst, \ldots, \IdxLast{p} $ \\
~~~$ \VrowT{T} \becomes \VrowT{T}_{j_{\IdxSecond}} =  \M{A} \VrowT{X}_{j_{\IdxSecond}} $ \\
~~~{\bf for } $ j_{\IdxFirst} = \IdxFirst, \ldots, j_{\IdxSecond} $ \\
~~~~~~ $ \gamma_{\subtwo{j_{\IdxFirst}}{j_{\IdxSecond}}} \becomes \Vrow{X}_{j_{\IdxFirst}} \VrowT{T} $ \\
~~~{\bf endfor} \\
{\bf endfor}
\end{tabular}
\end{minipage}
\\ \hline  
\multicolumn{2}{c}{}
\\[-0.1in]
\hline
\multicolumn{2}{|l|}{
$ \T{A} $ is a 3-way tensor ($ m = 3 $):
$ \T{C} \becomes [\T{A};\M{X},\M{X}, \M{X}] $
}
\\ \hline
\begin{minipage}{2.5in}
\begin{tabular}[c]{@{}l}
{\bf for} $ j_\IdxThird = \IdxFirst, \ldots, \IdxLast{p} $ \\
~~~{\bf for} $ j_\IdxSecond = \IdxFirst, \ldots, j_\IdxThird $ \\
~~~~~~{\bf for} $ j_\IdxFirst = \IdxFirst, \ldots, j_\IdxSecond$ \\
~~~~~~~~~$ \gamma_{\subthree{j_\IdxFirst}{j_\IdxSecond}{j_\IdxThird}} \becomes 0 $ \\
~~~~~~~~~{\bf for} $ i_\IdxThird = \IdxFirst, \ldots, \IdxLast{n} $ \\
~~~~~~~~~~~~{\bf for} $ i_\IdxSecond = \IdxFirst, \ldots, \IdxLast{n} $ \\
~~~~~~~~~~~~~~~{\bf for} $i_\IdxFirst = \IdxFirst, \ldots, \IdxLast{n} $ \\
~~~~~~~~~~~~~~~~~~$
\gamma_{\subthree{j_\IdxFirst}{j_\IdxSecond}{j_\IdxThird}} +\becomes
$ \\
~~~~~~~~~~~~~~~~~~~~~$ \alpha_{\subthree{i_\IdxFirst}{i_\IdxSecond}{i_\IdxThird}} \chi_{\subtwo{j_\IdxFirst}{i_\IdxFirst}}
 \chi_{\subtwo{j_\IdxSecond}{i_\IdxSecond}} \chi_{\subtwo{j_\IdxThird}{i_\IdxThird}} $ \\
~~~~~~~~~~~~~~~{\bf endfor} \\
~~~~~~~~~$ \iddots $ \\
{\bf endfor}
\end{tabular}
\end{minipage}
&
\begin{minipage}[c]{1.5in}
\begin{tabular}{@{}l}
{\bf for} $ j_\IdxThird = \IdxFirst, \ldots, \IdxLast{p} $ \\
~~~$ \M{T}^{(2)} \becomes \M{T}^{(2)}_{j_\IdxThird} =  \T{A} \times_\IdxThird \Vrow{X}_{j_\IdxThird} $ \\
~~~{\bf for} $ j_\IdxSecond = \IdxFirst, \ldots, j_\IdxThird $ \\
~~~~~~$ \VrowT{T}^{(1)} \becomes \VrowT{T}^{(1)}_{\subtwo{j_\IdxSecond}{j_\IdxThird}} =  \M{T}^{(2)} \times_\IdxSecond \Vrow{X}_{j_\IdxSecond} $ \\
~~~~~~{\bf for} $ j_\IdxFirst = \IdxFirst, \ldots, j_\IdxSecond $ \\
~~~~~~~~~$ \gamma_{\subthree{j_\IdxFirst}{j_\IdxSecond}{j_\IdxThird}} \becomes \VrowT{T}^{(1)} \times_\IdxFirst \Vrow{X}_{j_\IdxFirst} $ \\
~~~~~~{\bf endfor} \\
~~~{\bf endfor} \\
{\bf endfor}
\end{tabular}
\end{minipage}
\\ \hline
\multicolumn{2}{c}{}
\\[-0.1in]
\hline
\multicolumn{2}{|l|}{
$ \T{A} $ is an $m$-way tensor: $ \T{C} \becomes [\T{A};\M{X},\cdots,\M{X}] $
}
\\ \hline
\begin{minipage}{2.5in}
\begin{tabular}[c]{@{}l}
{\bf for} $ j_{\IdxLast{m}} = \IdxFirst, \ldots, \IdxLast{p} $ \\
~~~$ \ddots $ \\
~~~~~~{\bf for} $ j_\IdxFirst = \IdxFirst, \ldots, j_\IdxSecond$ \\
~~~~~~~~~$ \gamma_{\subthree{j_\IdxFirst}{\cdots}{j_{\IdxLast{m}}}} \becomes 0 $ \\
~~~~~~~~~{\bf for} $ i_{\IdxLast{m}} = \IdxFirst, \ldots, \IdxLast{n} $ \\
~~~~~~~~~~~~$ \ddots $ \\
~~~~~~~~~~~~~~~{\bf for} $i_\IdxFirst = \IdxFirst, \ldots, \IdxLast{n} $ \\
~~~~~~~~~~~~~~~~~~$ \gamma_{\subthree{j_\IdxFirst}{\cdots}{j_{\IdxLast{m}}}} +\becomes $\\
~~~~~~~~~~~~~~~~~~~~~~~$ \alpha_{\subthree{i_\IdxFirst}{\cdots}{i_\IdxThird}} \chi_{\subtwo{j_\IdxFirst}{i_\IdxFirst}}
\cdots \chi_{\subtwo{j_{\IdxLast{m}}}{i_{\IdxLast{m}}}} $ \\
~~~~~~~~~~~~~~~{\bf endfor} \\
~~~~~~~~~$ \iddots $ \\
{\bf endfor}
\end{tabular}
\end{minipage}
&
\begin{minipage}[c]{1.5in}
\begin{tabular}{@{}l}
{\bf for} $ j_{\IdxLast{m}} = \IdxFirst, \ldots, \IdxLast{p} $ \\
~~~$ \T{T}^{(m-1)} \becomes \T{T}_{j_{\IdxLast{m}}} =  \T{A} \times_{\IdxLast{m}} \Vrow{X}_{j_{\IdxLast{m}}} $ \\
~~~~~~$\ddots$ \\
~~~~~~{\bf for} $ j_\IdxSecond = \IdxFirst, \ldots, j_\IdxThird $ \\
~~~~~~~~~$ \VrowT{T}^{(1)} \becomes \VrowT{T}_{\subthree{j_\IdxSecond}{\cdots}{j_{\IdxLast{m}}}} =  \M{T}^{(2)} \times_\IdxSecond \Vrow{X}_{j_\IdxSecond} $ \\
~~~~~~~~~{\bf for} $ j_\IdxFirst = \IdxFirst, \ldots, j_\IdxSecond $ \\
~~~~~~~~~~~~$ \gamma_{\subthree{j_\IdxFirst}{\cdots}{j_{\IdxLast{m}}}} \becomes \VrowT{T}^{(1)} \times_\IdxFirst \Vrow{X}_{j_\IdxFirst}$ \\
~~~~~~~~~{\bf endfor} \\
~~~~~~{\bf endfor} \\
~~~~~~$\iddots$ \\
{\bf endfor}
\end{tabular}
\end{minipage}
\\ \hline
\end{tabular}
\end{center} 
\caption{Algorithms for $ \T{C} \becomes [ \T{A}; \M{X},
  \cdots , \M{X}] $ that compute with scalars.  In order to facilitate
  the comparing and contrasting of algorithms, we present algorithms
  for the special cases where $ m = 2, 3 $ (top and middle) as well as
  the general case (bottom).  For each, we give the naive algorithm on
the left and the algorithm that reduces computation at the expense of
temporary storage on the right.}
\label{fig:MatrixAlg}
\label{fig:3DAlg}
\end{figure}

%

Based on (\ref{eqn:matrix1}), a naive algorithm that only computes the upper
triangular part of symmetric matrix $\M{C} = \M{X} \M{A} \M{X}^T$ is given in
Figure~\ref{fig:MatrixAlg}~(top left), at a cost of approximately $3 p^2 n^2$
floating point operations (flops).
The algorithm to its right reduces flops by storing intermediate results and
taking advantage of symmetry.  It is motivated by observing that
\begin{alignat}{2} 
& \M{X} \M{A} \M{X}^T &&= \M{X} 
\begin{array}[t]{c}
\underbrace{\M{A} \M{X}^T} \\
\M{T} 
\end{array} \nonumber \\
&  &&=
\left( \begin{array}{c} 
\Vrow{X}_\IdxFirst \\
\vdots \\ 
\Vrow{X}_{\IdxLast{p}} 
\end{array}
\right) 
\begin{array}[t]{c}
\underbrace{
\M{A}
\left( \begin{array}{c c c} 
\VrowT{X}_\IdxFirst & \cdots & \VrowT{X}_{\IdxLast{p}}
\end{array}
\right)
} \\
\left( \begin{array}{c c c} 
\VrowT{T}_\IdxFirst & \cdots & \VrowT{T}_{\IdxLast{p}}
\end{array}
\right)
\end{array}
=
\left( \begin{array}{c} 
\Vrow{X}_\IdxFirst \\
\vdots \\ 
\Vrow{X}_{\IdxLast{p}} 
\end{array}
\right)
\left( \begin{array}{c c c} 
\VrowT{T}_\IdxFirst & \cdots & \VrowT{T}_{\IdxLast{p}}
\end{array}
\right), \nonumber \\
\end{alignat}
where $\VrowT{T}_j = \M{A} \VrowT{X}_j \in \mathbb{R}^{n}$ and $\VrowT{X}_{j} \in
\mathbb{R}^{n}$ (recall that $\VrowT{x}_j$ denotes the transpose of
the $j$th row of $\M{X}$).
This algorithm requires approximately
$ 2 pn^2 + p^2n$ flops at the expense of
requiring temporary space for a vector $\VrowT{T} $.

\subsection{\BCSS(\BCSSshort) for $m=2$}

Since matrices $\M{C}$ and $\M{A}$ are symmetric, it saves space to store only
the upper (or lower) triangular part of those matrices.
We consider storing the upper triangular part.
While for matrices the savings is modest (and rarely exploited), the savings is
more dramatic for tensors of higher order.

To store a symmetric matrix, consider packing the elements of the upper triangle
tightly into memory with the following ordering of unique elements:
\[
\begin{pmatrix}
0 & 1 & 3 & \cdots \\
  & 2 & 4 & \cdots \\
  &   & 5 & \cdots \\
  &   &   & \ddots 
\end{pmatrix}.
\]
Variants of this theme have been proposed over the
course of the last few decades but have never caught on due to the
complexity that is introduced when indexing the elements of the
matrix~\cite{Gunnels:2001:FFL,BaKoPl11}.  Given that this complexity only increases with the tensor
order, we do not pursue this idea.

%
\begin{figure}[tb!]
\begin{center}
\begin{tabular}{|@{\hspace{4pt}} l @{\hspace{4pt}} I @{\hspace{4pt}}l @{\hspace{4pt}}| c  | c | c | c |} \hline
Relative storage of \BCSSshort & & \multicolumn{4}{c|}{$\blkedDimA=\lceil n/\blkDim{\M{A}} \rceil$} \\ \cline{3-6}
& & 2 & 4 & 8 & 16 \\ 
& & 
{
\setlength{\unitlength}{0.15in}
\begin{picture}(2,2)
\put(1,0){\framebox(1,1){}}
\put(0,1){\framebox(1,1){}}
\put(1,1){\framebox(1,1){}}
\end{picture}
}
&
{
\setlength{\unitlength}{0.075in}
\begin{picture}(4,4)
\multiput(0,3)(1,0){4}{\framebox(1,1){}}
\multiput(1,2)(1,0){3}{\framebox(1,1){}}
\multiput(2,1)(1,0){2}{\framebox(1,1){}}
\multiput(3,0)(1,0){1}{\framebox(1,1){}}
\end{picture}
}
&
{
\setlength{\unitlength}{0.0375in}
\begin{picture}(8,8)
\multiput(0,7)(1,0){8}{\framebox(1,1){}}
\multiput(1,6)(1,0){7}{\framebox(1,1){}}
\multiput(2,5)(1,0){6}{\framebox(1,1){}}
\multiput(3,4)(1,0){5}{\framebox(1,1){}}
\multiput(4,3)(1,0){4}{\framebox(1,1){}}
\multiput(5,2)(1,0){3}{\framebox(1,1){}}
\multiput(6,1)(1,0){2}{\framebox(1,1){}}
\multiput(7,0)(1,0){1}{\framebox(1,1){}}
\end{picture}
}
&
{
\setlength{\unitlength}{0.01875in}
\begin{picture}(16,16)
\multiput(0,15)(1,0){16}{\framebox(1,1){}}
\multiput(1,14)(1,0){15}{\framebox(1,1){}}
\multiput(2,13)(1,0){14}{\framebox(1,1){}}
\multiput(3,12)(1,0){13}{\framebox(1,1){}}
\multiput(4,11)(1,0){12}{\framebox(1,1){}}
\multiput(5,10)(1,0){11}{\framebox(1,1){}}
\multiput(6,9)(1,0){10}{\framebox(1,1){}}
\multiput(7,8)(1,0){9}{\framebox(1,1){}}
\multiput(8,7)(1,0){8}{\framebox(1,1){}}
\multiput(9,6)(1,0){7}{\framebox(1,1){}}
\multiput(10,5)(1,0){6}{\framebox(1,1){}}
\multiput(11,4)(1,0){5}{\framebox(1,1){}}
\multiput(12,3)(1,0){4}{\framebox(1,1){}}
\multiput(13,2)(1,0){3}{\framebox(1,1){}}
\multiput(14,1)(1,0){2}{\framebox(1,1){}}
\multiput(15,0)(1,0){1}{\framebox(1,1){}}
\end{picture}
}
\\ \whline 
relative to minimal storage & $ \frac{n(n+1)/2}{n(n+\blkDim{\M{A}})/2} $ &
0.67 & 0.80 & 0.89 & 0.94 \\ \hline
relative to dense storage & $ \frac{n^2}{n(n+\blkDim{\M{A}})/2} $ &
1.33 & 1.60 & 1.78 & 1.88 \\ \whline
\end{tabular}
\end{center}
\caption{Storage savings factor of \BCSSshort when $n=512$.}
\label{fig:blockedStorageSavings}

\end{figure}


Instead, we embrace an idea, storage by blocks, that was introduced into the
{\tt libflame} library~\cite{FLAWN12,libflame_ref,CiSE09} in order to support
algorithms by blocks.
Submatrices (blocks) become units of data and operations with those blocks
become units of computation.
Partition the symmetric matrix  $\M{A} \in \matsz{n}{n}$ into submatrices as
\begin{equation}
  \label{eq:ACB}
\M{A} = \begin{pmatrix}
\M{A}_{\subtwo{\IdxFirst}{\IdxFirst}} & \M{A}_{\subtwo{\IdxFirst}{\IdxSecond}} & \M{A}_{\subtwo{\IdxFirst}{\IdxThird}} & \cdots & \M{A}_{\subtwo{\IdxFirst}{\IdxParenLast{\nB}}} \\
\color{grey} \M{A}_{\subtwo{\IdxSecond}{\IdxFirst}} & \M{A}_{\subtwo{\IdxSecond}{\IdxSecond}} & \M{A}_{\subtwo{\IdxSecond}{\IdxThird}} & \cdots & \M{A}_{\subtwo{\IdxSecond}{\IdxParenLast{\nB}}} \\
\color{grey} \M{A}_{\subtwo{\IdxThird}{\IdxFirst}} & \color{grey} \M{A}_{\subtwo{\IdxThird}{\IdxSecond}} & \M{A}_{\subtwo{\IdxThird}{\IdxThird}} & \cdots & 
\M{A}_{\subtwo{\IdxThird}{\IdxParenLast{\nB}}} \\
\vdots & \vdots & \vdots & \ddots & \vdots \\
\color{grey} \M{A}_{\subtwo{\IdxParenLast{\nB}}{\IdxFirst}} &\color{grey} \M{A}_{\subtwo{\IdxParenLast{\nB}}{\IdxSecond}} & \color{grey} \M{A}_{\subtwo{\IdxParenLast{\nB}}{\IdxThird}} & \cdots & \M{A}_{\subtwo{\IdxParenLast{\nB}}{\IdxParenLast{\nB}}} 
\end{pmatrix}.
\end{equation}
Here each submatrix $\M{A}_{\subtwo{\blkI_\IdxFirst}{\blkI_\IdxSecond}} \in \mathbb{R}^{\bA \times \bA}$.
We define $\nB = n/\bA$ where, without loss of generality, we assume $\bA$ evenly divides $n$.
Hence $\M{A}$ is a blocked $\nB \times \nB$ matrix with
blocks of size $\bA \times \bA$.  
The blocks are stored using some conventional method (e.g., each
$\M{A}_{\subtwo{\blkI_\IdxFirst}{\blkI_\IdxSecond}}$ is stored in column-major
order).
For symmetric matrices, the blocks below the diagonal are redundant and need not
be stored (indicated by gray coloring). We do not store the data these blocks
represent explicitly; instead, we store information at these locations informing
us how to obtain the required data.  By doing this, we can retain a simple
indexing scheme into $\M{A}$ that avoids the complexity associated with storing
only the unique entries.
Although the diagonal blocks are themselves symmetric, we do not take advantage
of this in order to simplify the access pattern for the computation with those
blocks.
We refer to this storage technique as \BCSS(\BCSSshort) throughout the rest of
this paper.

Storing the upper triangular \emph{individual elements} of the symmetric matrix
$\M{A}$ requires storage of \[ n(n+1)/2 = \binom{n+1}{2} ~\rm floats.
\] In constrast, storing the upper triangular \emph{blocks} of the symmetric
matrix $\M{A}$ with \BCSSshort requires \[ n(n+\bA)/2 = \bA^2 \binom{\nB+1}{2}
~\rm floats.
\] The \BCSSshort scheme requires a small amount of additional storage,
depending on $\bA$.
\Fig{blockedStorageSavings} illustrates how the storage for \BCSSshort
approaches the cost of storing only the upper triangular elements (here $n = 512
$) as the number of blocks increases.

\subsection{Algorithm-by-blocks for $m=2$}

Given that $\M{C}$ and $\M{A}$ are stored with \BCSSshort, we now need to
discuss how the algorithm computes with these blocks.
Partition $\M{A}$ as in \Eqn{ACB},
\begin{displaymath}
  \M{C} = \begin{pmatrix}
    \M{C}_{\subtwo{\IdxFirst}{\IdxFirst}} & \cdots & \M{C}_{\subtwo{\IdxFirst}{\IdxParenLast{\pB}}} \\
    \vdots & \ddots & \vdots \\
    \color{grey} \M{C}_{\subtwo{\IdxParenLast{\pB}}{\IdxFirst}}
    & \cdots & 
    \M{C}_{\subtwo{\IdxParenLast{\pB}}{\IdxParenLast{\pB}}} 
  \end{pmatrix}
  \mbox{, and }
  \M{X} = \begin{pmatrix}
    \M{X}_{\subtwo{\IdxFirst}{\IdxFirst}} & \cdots & \M{X}_{\subtwo{\IdxFirst}{\IdxParenLast{\nB}}} \\
    \vdots & \ddots & \vdots \\
    \M{X}_{\subtwo{\IdxParenLast{\pB}}{\IdxFirst}} & \cdots & \M{X}_{\subtwo{\IdxParenLast{\pB}}{\IdxParenLast{\nB}}} 
  \end{pmatrix}.
\end{displaymath}
Without loss of generality, $\pB = p / \bC$ is integral, and 
the blocks of $\M{C}$ and $\M{X}$ are of size $\bC \times \bC$ and
$\bC \times \bA$, respectively.
Then $\M{C} \becomes \M{X} \M{A} \M{X}^T$ means that
\begin{align}
\label{eqn:blocked_2D}
\M{C}_{\subtwo{\blkJ_{\IdxFirst}}{\blkJ_{\IdxSecond}}} &=& 
\left(\begin{array}{c c c c c}
\M{X}_{\subtwo{\blkJ_{\IdxFirst}}{\IdxFirst}} & \cdots & \M{X}_{\subtwo{\blkJ_{\IdxFirst}}{\IdxParenLast{\nB}}} 
\end{array} \right)
\left(\begin{array}{c c c c c}
\M{A}_{\subtwo{\IdxFirst}{\IdxFirst}} & \cdots & \M{A}_{\subtwo{\IdxFirst}{\IdxParenLast{\nB}}} \\
\vdots & \ddots & \vdots \\
\color{grey} \M{A}_{\subtwo{\IdxParenLast{\nB}}{\IdxFirst}} & \cdots & \M{A}_{\subtwo{\IdxParenLast{\nB}}{\IdxParenLast{\nB}}} 
\end{array} \right)
\left(\begin{array}{c}
\M{X}_{\subtwo{\blkJ_{\IdxSecond}}{\IdxFirst}}^T \\
\vdots \\
\M{X}_{\subtwo{\blkJ_{\IdxSecond}}{\IdxParenLast{\nB}}}^T 
\end{array} \right) \nonumber
\\
&=&
\sum_{\blkI_{\IdxFirst}=\IdxFirst}^{\IdxLast{\nB}}
\sum_{\blkI_{\IdxSecond}=\IdxFirst}^{\IdxLast{\nB}} \M{X}_{\subtwo{\blkJ_{\IdxFirst}}{\blkI_{\IdxFirst}}} \M{A}_{\subtwo{\blkI_{\IdxFirst}}{\blkI_{\IdxSecond}}} \M{X}_{\subtwo{\blkJ_{\IdxSecond}}{\blkI_{\IdxSecond}}}^T 
= 
\sum_{\blkI_{\IdxFirst}=\IdxFirst}^{\IdxLast{\nB}}
\sum_{\blkI_{\IdxSecond}=\IdxFirst}^{\IdxLast{\nB}} [\M{A}_{\subtwo{\blkI_{\IdxFirst}}{\blkI_{\IdxSecond}}}; \M{X}_{\subtwo{\blkJ_{\IdxFirst}}{\blkI_{\IdxFirst}}}, \M{X}_{\subtwo{\blkJ_{\IdxSecond}}{\blkI_{\IdxSecond}}} ] \mbox{~(in tensor
  notation)}.
\end{align}
This yields the algorithm in \Fig{2D}, in which an analysis of its cost is also
given.
This algorithm avoids redundant computation, except within symmetric blocks on
the diagonal.
Comparing (\ref{eqn:blocked_2D}) to (\ref{eqn:matrix1}), we see that the only
difference lies in replacing scalar terms with their block counterparts. 
Consequently, comparing this algorithm with the one in
\Fig{MatrixAlg}~(top right), we notice that every scalar has simply been
replaced by a block (submatrix).  The algorithm now computes a
temporary matrix $\M{T} = \M{A} \M{X}_{\blkJ_{\IdxSecond}:}^T$ instead of a
temporary vector $\VrowT{t} = \M{A}\VrowT{x}_{j_\IdxSecond}$ as in
\Fig{MatrixAlg} for each index $\IdxFirst \le \blkJ_{\IdxSecond} < \blkedDimC$.
It requires $n  \times \bC$ extra storage instead of $n$ extra storage in addition to
the storage for $\M{C}$ and $\M{A} $.

%
%
%
%
%
%
%

%
%
%

%
%
%
%

\begin{figure}[tbp!]
{%
\setlength{\arraycolsep}{2pt}
\begin{center}
\footnotesize
\begin{tabular}{| @{\hspace{3pt}} l @{\hspace{3pt}}| c | c | c |} \hline
Algorithm & Ops & Total \# of & Temp. \\
& (flops) & times executed & storage \\ \hline \hline
{\bf for} $ \blkJ_{\IdxSecond} = \IdxFirst, \ldots, \IdxLast{\blkedDimC} $  & & & \\
\tiny
~~~~~~$ 
\begin{array}[t]{c}
\underbrace{
\left(\begin{array}{c} 
\M{T}_{\IdxFirst}^T \\
\vdots \\ 
\M{T}_{\IdxLast{\blkedDimA}}^T 
\end{array} \right)
}
\\
\M{T}
\end{array}
\becomes 
\begin{array}[t]{c}
\underbrace{
\left(\begin{array}{c c c c c}
\M{A}_{\subtwo{\IdxFirst}{\IdxFirst}} & \cdots & \M{A}_{\subtwo{\IdxFirst}{\IdxParenLast{\blkedDimA}}} \\
\vdots & \ddots & \vdots \\
\color{grey} \M{A}_{\subtwo{\IdxParenLast{\blkedDimA}}{\IdxFirst}} & \cdots & \M{A}_{\subtwo{\IdxParenLast{\blkedDimA}}{\IdxParenLast{\blkedDimA}}} 
\end{array} \right)
} \\
\M{A}
\end{array}
\begin{array}[t]{c}
\underbrace{
\left(\begin{array}{c} 
\M{X}_{\subtwo{\blkJ_{\IdxSecond}}{\IdxFirst}}^T \\
\vdots \\ 
\M{X}_{\subtwo{\blkJ_{\IdxSecond}}{\IdxParenLast{\blkedDimA}}}^T 
\end{array} \right)
}
\\
\M{X}_{\blkJ_{\IdxSecond}}^T
\end{array}
$ &
$ 2 \blkDim{\M{C}} n^2 $ & $ \blkedDimC $ & $\blkDim{\M{C}} n$
\\[0.3in]
~~~~~~{\bf for} $ \blkJ_{\IdxFirst} = \IdxFirst, \ldots , \blkJ_{\IdxSecond} $ & & & \\
~~~~~~~~~~~~$
\M{C}_{\subtwo{\blkJ_{\IdxFirst}}{\blkJ_{\IdxSecond}}} 
\becomes
\left(\begin{array}{c c c c c}
\M{X}_{\subtwo{\blkJ_{\IdxFirst}}{\IdxFirst}} & \cdots & \M{X}_{\subtwo{\blkJ_{\IdxFirst}}{\IdxParenLast{\blkedDimA}}} 
\end{array} \right)
\left(\begin{array}{c} 
\M{T}_{\IdxFirst}^T \\
\vdots \\ 
\M{T}_{\IdxLast{\blkedDimA}}^T 
\end{array} \right)
$ 
& $ 2 \blkDim{\M{C}} ^2 n $  & $ \blkedDimC ( \blkedDimC +1 )/2 $ &
\\
~~~~~~{\bf endfor} & & & \\
{\bf endfor} & & &
\\ \hline
\multicolumn{4}{| c}{}
\\[-0.1in]
\hline
\multicolumn{4}{| r |}{
\begin{minipage}{4.75in}
\[
\begin{array}{l}
\mbox{Total Cost:~~} 
2 \blkDim{\M{C}} n^2 \blkedDimC + 
2 \blkDim{\M{C}}^2 n \left(\blkedDimC(\blkedDimC+1)/2\right)  =
\sum_{d=0}^1%
\left( 
2 \blkDim{\M{C}}^{d+1} n^{2-d} 
\binom{\blkedDimC +d}{d+1}
\right)
\approx
2 p n^2 + 
p^2 n
\mbox{~flops}
\\
\mbox{Total temporary storage:~~}
\blkDim{\M{C}}n
=
\sum_{d=0}^0%
\left( 
\blkDim{\M{C}}^{d+1} n^{1-d}
\right)
\mbox{~entries}
\end{array}
\]
\end{minipage}}
\\ \hline
\end{tabular}
\end{center}%
}%
\caption{Algorithm-by-blocks for computing $ \M{C} \becomes \M{X} \M{A} \M{X}^T =
[ \M{A}; \M{X}, \M{X} ] $.  The algorithm assumes that $ \M{C} $ is
partitioned into blocks of size $ \blkDim{\M{C}} \times \blkDim{\M{C}} $, with $ \blkedDimC = \lceil p/ \blkDim{\M{C}}
\rceil $. An expression using summations is given to help in identifying a pattern later on.}
\label{fig:2D}

{%
\setlength{\arraycolsep}{2pt}
\begin{center}
\footnotesize
\begin{tabular}{|@{\hspace{5pt}} l @{\hspace{5pt}}| c | c | c |} \hline
Algorithm & Ops & Total \# of & Temp. \\
& (flops) & times executed & storage \\ \hline \hline
{\bf for} $ \blkJ_\IdxThird = \IdxFirst, \ldots, \IdxLast{\blkedDimC} $ & & & \\ 
~~~~~~  $  
\left( \begin{array}{c c c}
\T{T}^{(2)}_{\subtwo{\IdxFirst}{\IdxFirst}} & \cdots & \T{T}^{(2)}_{\subtwo{\IdxFirst}{\IdxParenLast{\blkedDimA}}} \\
\vdots & \ddots & \vdots \\
\T{T}^{(2)}_{\subtwo{\IdxParenLast{\blkedDimA}}{\IdxFirst}} & \cdots & \T{T}^{(2)}_{\subtwo{\IdxParenLast{\blkedDimA}}{\IdxParenLast{\blkedDimA}}} \\
\end{array} \right)
\becomes $ & & & \\
~~~~~~~~~~~~~~~~~~~~~~  $  
\T{A} \times_\IdxThird
\left(\begin{array}{c c c c c}
\M{X}_{\subtwo{\blkJ_\IdxThird}{\IdxFirst}} & \cdots & \M{X}_{\subtwo{\blkJ_\IdxThird}{\IdxParenLast{\blkedDimA}}} 
\end{array} \right)
$ 
& 
$ 2 \blkDim{\T{C}} n^3 $ & $ \blkedDimC $ & $\blkDim{\T{C}} n^2$
\\
~~~~~~{\bf for} $ \blkJ_\IdxSecond = \IdxFirst, \ldots , \blkJ_\IdxThird $ & & & \\
~~~~~~~~~~~~$ 
\left(
\begin{array}{c}
\T{T}^{(1)}_\IdxFirst \\
\vdots \\
\T{T}^{(1)}_{\IdxLast{\blkedDimA}}
\end{array}\right)
\becomes 
\left( \begin{array}{c c c}
\T{T}^{(2)}_{\subtwo{\IdxFirst}{\IdxFirst}} & \cdots & \T{T}^{(2)}_{\subtwo{\IdxFirst}{\IdxParenLast{\blkedDimA}}} \\
\vdots & \ddots & \vdots \\
\T{T}^{(2)}_{\subtwo{\IdxParenLast{\blkedDimA}}{\IdxFirst}} & \cdots & \T{T}^{(2)}_{\subtwo{\IdxParenLast{\blkedDimA}}{\IdxParenLast{\blkedDimA}}} \\
\end{array} \right) $ 
& 
$2 \blkDim{\T{C}}^{2} n^2$
&
$
\begin{array}{c}
\blkedDimC (\blkedDimC +1)/2 \\
= \\
\binom{\blkedDimC + 1}{2}
\end{array}
$
&
$\blkDim{\T{C}}^2 n$
\\
$ ~~~~~~~~~~~~~~~~~~~~~~~~~~~~~~~~~~~~\times_\IdxSecond
\left(\begin{array}{c c c} 
\M{X}_{\subtwo{\blkJ_\IdxSecond}{\IdxFirst}} & \cdots & \M{X}_{\subtwo{\blkJ_\IdxSecond}{\IdxParenLast{\blkedDimA}}} 
\end{array} \right)
$ 
& & &
\\
~~~~~~~~~~~~{\bf for} $ \blkJ_\IdxFirst = \IdxFirst, \ldots , \blkJ_\IdxSecond $ & & & \\
~~~~~~~~~~~~~~~~~~  $ 
\T{C}_{\subthree{\blkJ_\IdxFirst}{\blkJ_\IdxSecond}{\blkJ_\IdxThird}}
\becomes $ & & & \\
~~~~~~~~~~~~~~~~~~~~ $
\left(
\begin{array}{c}
\T{T}^{(1)}_\IdxFirst \\
\vdots \\
\T{T}^{(1)}_{\IdxLast{\blkedDimA}}
\end{array}\right)
\times_\IdxFirst
\left(\begin{array}{c c c} 
\M{X}_{\subtwo{\blkJ_\IdxFirst}{\IdxFirst}} & \cdots & \M{X}_{\subtwo{\blkJ_\IdxFirst}{\IdxParenLast{\blkedDimA}}} 
\end{array} \right)
$
& 
$2 \blkDim{\T{C}}^{3} n$
&
$
\begin{array}{c}
\frac{\blkedDimC (\blkedDimC +1)(\blkedDimC +2)}{6} = \\
\binom{\blkedDimC + 2}{3}
\end{array}
$
&
\\
~~~~~~~~~~~~{\bf endfor} & & &\\
~~~~~~{\bf endfor} & & &\\
{\bf endfor} & & &
\\ \hline
\multicolumn{4}{|c}{}
\\[-0.1in]
\hline
\multicolumn{4}{| r |}{
\begin{minipage}{4.75in}
\[
\begin{array}{l}
\mbox{Total Cost:~~} 
\sum_{d=0}^2%
\left( 
2 \blkDim{\T{C}}^{d+1} n^{3-d} 
\binom{\blkedDimC +d}{d+1}
\right)
\approx 
2 p n^3 + p^2 n^2 + \frac{p^3 n}{3}
\mbox{~flops}
\\
\mbox{Total temporary storage:~~} 
\blkDim{\T{C}}n^2 + \blkDim{\T{C}}^2n
 =
\sum_{d=0}^1%
\left( 
\blkDim{\T{C}}^{d+1} n^{2-d}
\right)
\mbox{~entries}
\end{array}
\]
\end{minipage}}
\\ \hline
\end{tabular}
\end{center}%
}%
\caption{Algorithm-by-blocks for computing $ [ \T{A}; \M{X}, \M{X}, \M{X} ] $.  The algorithm assumes that $ \T{C} $ is
partitioned into blocks of size $ \blkDim{\T{C}} \times \blkDim{\T{C}} \times \blkDim{\T{C}}  $, with $ \blkedDimC = \lceil p/ \blkDim{\T{C}}
\rceil $. An expression using summations is given to help in identifying a pattern later on.}
\label{fig:3D}
\end{figure}


%
%
%
%

\section{The 3-way Case}

We extend the insight gained in the last section to the case where $ \T{C} $ and
$ \T{A} $ are symmetric order-$3$ tensors before moving on to the general
order-$m$ case in the next section.

\subsection{The operation for $m=3$}

Now $ \T{C} \becomes [ \T{A}; \M{X}, \M{X}, \M{X} ] $ where
$ \T{A} \in
\tensz{m}{n} $, $ \T{C} \in \tensz{m}{p} $, and $ [ \T{A}; \M{X}, \M{X}, \M{X} ] = \T{A}
\times_\IdxFirst \M{X} \times_\IdxSecond \M{X} \times_\IdxThird \M{X} $.
In 
our discussion, $ \T{A}$ is a symmetric tensor, as is $\T{C}$ by virtue of the
operation applied to $\T{A}$.
Now, \[
\begin{array}{l c l}
\gamma_{\subthree{j_\IdxFirst}{j_\IdxSecond}{j_\IdxThird}} & = & \T{A} \times_\IdxFirst \Vrow{X}_{j_\IdxFirst} \times_\IdxSecond \Vrow{X}_{j_\IdxSecond} \times_\IdxThird \Vrow{X}_{j_\IdxThird} 
= \sum_{i_\IdxFirst=\IdxFirst}^{\IdxLast{n}} ( \T{A} \times_\IdxSecond \Vrow{X}_{j_\IdxSecond} \times_\IdxThird \Vrow{X}_{j_\IdxThird} )_{i_\IdxFirst} \times_\IdxFirst \chi_{\subtwo{j_\IdxFirst}{i_\IdxFirst}} \\
& = & \sum_{i_\IdxFirst=\IdxFirst}^{\IdxLast{n}} ( \sum_{i_\IdxSecond=\IdxFirst}^{\IdxLast{n}} (\T{A} \times_\IdxThird
\Vrow{X}_{j_\IdxThird})_{i_\IdxSecond} \times_\IdxSecond \chi_{\subtwo{j_\IdxSecond}{i_\IdxSecond}} )_{i_\IdxFirst} \times_\IdxFirst
\chi_{\subtwo{j_\IdxFirst}{i_\IdxFirst}} \\  
&=& \sum_{i_\IdxThird=\IdxFirst}^{\IdxLast{n}} \sum_{i_\IdxSecond=\IdxFirst}^{\IdxLast{n}} \sum_{i_\IdxFirst=\IdxFirst}^{\IdxLast{n}} \alpha_{\subthree{i_\IdxFirst}{i_\IdxSecond}{i_\IdxThird}} \chi_{\subtwo{j_\IdxFirst}{i_\IdxFirst}} \chi_{\subtwo{j_\IdxSecond}{i_\IdxSecond}} \chi_{\subtwo{j_\IdxThird}{i_\IdxThird}} .
\end{array}
\]

\subsection{Simple algorithms for $m=3$}

A naive algorithm is given in Figure~\ref{fig:3DAlg}~(middle left). The cheaper
algorithm to its right is motivated by {
\setlength{\arraycolsep}{2pt}
\begin{eqnarray*}
\lefteqn{
\T{A} \times_\IdxFirst \M{X} \times_\IdxSecond \M{X} \times_\IdxThird \M{X} =
\T{A} \times_\IdxFirst
\left( \begin{array}{c} 
\Vrow{X}_\IdxFirst \\
\vdots \\ 
\Vrow{X}_{\IdxLast{p}} 
\end{array}
\right)
\times_\IdxSecond
\left( \begin{array}{c} 
\Vrow{X}_\IdxFirst \\
\vdots \\ 
\Vrow{X}_{\IdxLast{p}} 
\end{array}
\right)
\times_\IdxThird
\left( \begin{array}{c} 
\Vrow{X}_\IdxFirst \\
\vdots \\ 
\Vrow{X}_{\IdxLast{p}} 
\end{array}
\right)} \\
& = &
\begin{array}[t]{c}
\underbrace{
\left( \begin{array}{c c c} 
\M{T}_{\IdxFirst}^{(2)} & \cdots & \M{T}_{\IdxLast{p}}^{(2)} \\
\end{array}
\right)} \\
\T{T}^{(2)}
\end{array}
\times_\IdxFirst
\left( \begin{array}{c} 
\Vrow{X}_\IdxFirst \\
\vdots \\ 
\Vrow{X}_{\IdxLast{p}} 
\end{array}
\right)
\times_\IdxSecond
\left( \begin{array}{c} 
\Vrow{X}_\IdxFirst \\
\vdots \\ 
\Vrow{X}_{\IdxLast{p}} 
\end{array}
\right) \\
& = &
\begin{array}[t]{c}
\underbrace{
\left( \begin{array}{c c c} 
\V{T}^{(1)}_{\subtwo{\IdxFirst}{\IdxFirst}} & \cdots & \V{T}^{(1)}_{\subtwo{\IdxFirst}{\IdxParenLast{p}}} \\
\vdots & \ddots & \vdots \\ 
\V{T}^{(1)}_{\subtwo{\IdxParenLast{p}}{\IdxFirst}} & \cdots & \V{T}^{(1)}_{\subtwo{\IdxParenLast{p}}{\IdxParenLast{p}}} \\
\end{array}
\right)
} \\
\T{T}^{(1)}
\end{array}
\times_\IdxFirst
\left( \begin{array}{c} 
\Vrow{X}_\IdxFirst \\
\vdots \\ 
\Vrow{X}_{\IdxLast{p}} 
\end{array}
\right),
\end{eqnarray*}%
}%
where 

\fromtoMartin{}{
\[
\begin{array}{l c l}
\M{T}_{i_\IdxThird}^{(2)} \in \mathbb{R}^{n
\times n} & \text{and} & \M{T}_{i_\IdxThird}^{(2)} = \T{A} \times_\IdxThird
\Vrow{X}_{i_\IdxThird}, \\ \\
\V{T}^{(1)}_{\subtwo{i_\IdxSecond}{i_\IdxThird}} \in
\mathbb{R}^{n} & \text{and} & \V{T}^{(1)}_{\subtwo{i_\IdxSecond}{i_\IdxThird}} =
\M{T}_{i_\IdxThird}^{(2)} \times_\IdxSecond \Vrow{X}_{i_\IdxSecond} = \T{A}
\times_\IdxThird \Vrow{X}_{i_\IdxThird} \times_\IdxSecond
\Vrow{X}_{i_\IdxSecond}, \\ \\
\Vrow{X}_{j} \in \mathbb{R}^{1 \times n}. \\
\end{array}
\]
}

\fromtoMartin{
Although it appears $\T{T}^{(2)}$ is a vector (oriented
across the page) of matrices, it is actually a vector (oriented \emph{into} the
page) of matrices (each $\M{T}^{(2)}_{i_\IdxThird}$ should be viewed a matrix
oriented first down and then across the page).
Similarly, $\T{T}^{(1)}$ should be viewed as a matrix (first oriented
\emph{into} the page, then across the page) of vectors (each
$\V{T}^{(1)}_{i_\IdxSecond i_\IdxThird}$ should be viewed as a vector oriented
down the page).
This is a result of the mode-multiplication and is an aspect that is difficult
to represent on paper.}{}

This algorithm requires $p(2n^3 + p(2n^2 + 2pn)) = 2pn^3 + 2p^2n^2 + 2p^3n$
flops at the expense of requiring workspace for a matrix $\M{T}$ of size $n
\times n$ and vector $ \V{T} $ of length $n$.

\subsection{\BCSSshort for $m=3$}
In the matrix case (Section~\ref{sec:MatrixCase}), we described \BCSSshort, which stores only the blocks in the
upper triangular part of the matrix.  The storage
scheme used in the 3-way case is analogous to the matrix case; the difference
is that instead of storing blocks belonging
to a 2-way upper triangle, we must store the blocks in the ``upper triangular''
region of a 3-way tensor.  This region is comprised of all indices $
(i_\IdxFirst,i_\IdxSecond,i_\IdxThird) $ where $ i_\IdxFirst \le i_\IdxSecond
\le i_\IdxThird $.
For lack of a better term, we refer to this as \emph{upper hypertriangle} of the
tensor.

Similar to how we extended the notion of the upper triangular region of a 3-way
tensor, we must extend the notion of a block to three dimensions.  Instead of a
block being a two-dimensional submatrix, a block for 3-way tensors becomes a
3-way subtensor.
Partition tensor $\T{A} \in \tensz{3}{n}$ into cubical blocks 
$\T{A}_{\subthree{\blkI_\IdxFirst}{\blkI_\IdxSecond}{\blkI_\IdxThird}}$ of size
$\blkDim{\T{A}} \times \blkDim{\T{A}} \times \blkDim{\T{A}}$:
{
\setlength{\arraycolsep}{1pt}
\[
\begin{array}{l}
\T{A}_{\subthree{:}{:}{\IdxFirst}} = \left(\begin{array}{c c c c} 
\T{A}_{\subthree{\IdxFirst}{\IdxFirst}{\IdxFirst}} & \color{grey} \T{A}_{\subthree{\IdxFirst}{\IdxSecond}{\IdxFirst}} & \cdots & \color{grey} \T{A}_{\subthree{\IdxFirst}{\IdxParenLast{\blkedDimA}}{\IdxFirst}} \\
\color{grey} \T{A}_{\subthree{\IdxSecond}{\IdxFirst}{\IdxFirst}} & \color{grey} \T{A}_{\subthree{\IdxSecond}{\IdxSecond}{\IdxFirst}} & \cdots & \color{grey} \T{A}_{\subthree{\IdxSecond}{\IdxParenLast{\blkedDimA}}{\IdxFirst}} \\ 
\vdots & \vdots & \ddots & \vdots \\
\color{grey} \T{A}_{\subthree{\IdxParenLast{\blkedDimA}}{\IdxFirst}{\IdxFirst}} &\color{grey} \T{A}_{\subthree{\IdxParenLast{\blkedDimA}}{\IdxSecond}{\IdxFirst}} & \cdots & \color{grey} \T{A}_{\subthree{\IdxParenLast{\blkedDimA}}{\IdxParenLast{\blkedDimA}}{\IdxFirst}} 
\end{array} \right),
\cdots, \\
\T{A}_{\subthree{:}{:}{\IdxParenLast{\blkedDimA}}} = 
 \left(\begin{array}{c c c c} 
\T{A}_{\subthree{\IdxFirst}{\IdxFirst}{\IdxParenLast{\blkedDimA}}} & \T{A}_{\subthree{\IdxFirst}{\IdxSecond}{\IdxParenLast{\blkedDimA}}} & \cdots & \T{A}_{\subthree{\IdxFirst}{\IdxParenLast{\blkedDimA}}{\IdxParenLast{\blkedDimA}}} \\
\color{grey} \T{A}_{\subthree{\IdxSecond}{\IdxFirst}{\IdxParenLast{\blkedDimA}}} & \T{A}_{\subthree{\IdxSecond}{\IdxSecond}{\IdxParenLast{\blkedDimA}}} & \cdots & \T{A}_{\subthree{\IdxSecond}{\IdxParenLast{\blkedDimA}}{\IdxParenLast{\blkedDimA}}} \\
\vdots & \vdots & \ddots & \vdots \\
\color{grey} \T{A}_{\subthree{\IdxParenLast{\blkedDimA}}{\IdxFirst}{\IdxParenLast{\blkedDimA}}} &\color{grey} \T{A}_{\subthree{\IdxParenLast{\blkedDimA}}{\IdxSecond}{\IdxParenLast{\blkedDimA}}} & \cdots & \T{A}_{\subthree{\IdxParenLast{\blkedDimA}}{\IdxParenLast{\blkedDimA}}{\IdxParenLast{\blkedDimA}}} 
\end{array} \right),
\end{array}
\]%
}%
where $ \blkedDimA = n/\blkDim{\T{A}} $ (w.l.o.g. assume $\blkDim{\T{A}}$ divides $n$).
These blocks are stored using some conventional method 
and the blocks lying outside
the upper hypertriangular region are not stored.  Once again, we do not
take advantage of any symmetry within blocks (blocks with $\blkI_\IdxFirst =
\blkI_\IdxSecond$, $\blkI_\IdxFirst = \blkI_\IdxThird$, or $\blkI_\IdxSecond = \blkI_\IdxThird$) to
simplify the access pattern when computing with these blocks.

\begin{table}[tb!]
\centering
\begin{tabular}{| l I c | c | c |}
\hline
 & {Compact (Minimum)}
 & {Blocked Compact (\BCSSshort)}
 &  {Dense}\\ \whline
 $m=2$ & $
\frac{(n+1)n}{2} = \binom{n + 1}{2}$ & $\blkDim{\T{A}}^2\binom{\blkedDimA + 1}{2}$ & $n^2$ \\ \hline
 $m=3$ & $\binom{n + 2}{3}$ & $\blkDim{\T{A}}^3\binom{\blkedDimA + 2}{3}$ & $n^3$ \\ \hline
 $m=d$ & $\binom{n+d-1}{d}$ & $\blkDim{\T{A}}^d\binom{\blkedDimA + d-1}{d}$ & $n^d$ \\ \hline
\end{tabular}
\caption{Storage requirements for a tensor $\T{A}$ under different storage schemes.}
\label{tbl:storage_requirements}
\end{table}

As summarized in \Tbl{storage_requirements}, we see that while storing
{\em only} the upper hypertriangular elements of the tensor $ \T{A} $ requires $
\binom{n + 2}{3} $ storage, while \BCSSshort requires $
\blkDim{\T{A}}^3\binom{\blkedDimA + 2}{3} $ elements.  However, since $
\binom{\blkedDimA + 2}{3}\blkDim{\T{A}}^3 \approx
\frac{\blkedDimA^3}{3!}\blkDim{\T{A}}^3 = \frac{n^3}{6}$, we achieve a savings
of approximately a factor $ 6 $ if $ \blkedDimA $ is large enough, relative to
storing all elements.
Once again, we can apply the same storage method to $\T{C}$ for additional
savings.

Blocks such that $\blkI_\IdxFirst = \blkI_\IdxSecond \neq \blkI_\IdxThird$,
$\blkI_\IdxFirst = \blkI_\IdxThird \neq \blkI_\IdxSecond$, or $\blkI_\IdxSecond
= \blkI_\IdxThird \neq \blkI_\IdxFirst$ still have some symmetry and so
 are referred to as \emph{partially}-symmetric blocks.

\subsection{Algorithm-by-blocks for $m=3$}

We now discuss an algorithm-by-blocks for the 3-way case.
Partition $ \T{C} $ and $ \T{A} $ into blocks of size $ \blkDim{\T{C}} \times
\blkDim{\T{C}} \times \blkDim{\T{C}}$ and $ \blkDim{\T{A}} \times \blkDim{\T{A}}
\times \blkDim{\T{A}}$, respectively, and partition $ \M{X}$ into $
\blkDim{\T{C}} \times \blkDim{\T{A}} $ blocks.
Then, extending the insights we gained from the matrix case, $ \T{C} \becomes
[\T{A};\M{X},\M{X},\M{X}] $ means that
\begin{eqnarray*}
\T{C}_{\subthree{\blkJ_\IdxFirst}{\blkJ_\IdxSecond}{\blkJ_\IdxThird}} &=& 
\sum_{\blkI_\IdxFirst=\IdxFirst}^{\IdxLast{\blkedDimA}}
\sum_{\blkI_\IdxSecond=\IdxFirst}^{\IdxLast{\blkedDimA}}
\sum_{\blkI_\IdxThird=\IdxFirst}^{\IdxLast{\blkedDimA}} \T{A}_{\subthree{\blkI_\IdxFirst}{\blkI_\IdxSecond}{\blkI_\IdxThird}} \times_\IdxFirst \M{X}_{\subtwo{\blkJ_\IdxFirst}{\blkI_\IdxFirst}}
\times_\IdxSecond \M{X}_{\subtwo{\blkJ_\IdxSecond}{\blkI_\IdxSecond}} \times_\IdxThird \M{X}_{\subtwo{\blkJ_\IdxThird}{\blkI_\IdxThird}} \\
&=&
\sum_{\blkI_\IdxFirst=\IdxFirst}^{\IdxLast{\blkedDimA}}
\sum_{\blkI_\IdxSecond=\IdxFirst}^{\IdxLast{\blkedDimA}}
\sum_{\blkI_\IdxThird=\IdxFirst}^{\IdxLast{\blkedDimA}} [\T{A}_{\subthree{\blkI_\IdxFirst}{\blkI_\IdxSecond}{\blkI_\IdxThird}}; \M{X}_{\subtwo{\blkJ_\IdxFirst}{\blkI_\IdxFirst}},
\M{X}_{\subtwo{\blkJ_\IdxSecond}{\blkI_\IdxSecond}}, \M{X}_{\subtwo{\blkJ_\IdxThird}{\blkI_\IdxThird}} ]. %
\end{eqnarray*}
This yields the algorithm in Figure~\ref{fig:3D}, in which an analysis
of its cost is also given.
This algorithm avoids redundant computation, except for within blocks of $\T{C}$
that are symmetric or partially symmetric.  The algorithm computes temporaries
$\T{T}^{(2)} = \T{A} \times_{\IdxThird} \M{X}_{\blkJ_{\IdxThird}:}$ and
$\T{T}^{(1)} = \T{T}^{(2)} \times_{\IdxSecond} \M{X}_{\blkJ_{\IdxSecond}:}$ for
each index where $\IdxFirst \le \blkJ_{\IdxThird} < \blkedDimC$ and $\IdxFirst
\le \blkJ_{\IdxSecond} \le \blkJ_{\IdxThird}$.
The algorithm requires $ \blkDim{\T{C}}n^2 +  \blkDim{\T{C}}^2 n$ extra
storage (for $\T{T}^{(2)}$ and $\T{T}^{(1)}$, respectively), in addition to the
storage for $ \T{C} $ and $ \T{A} $.

%
%
%
%

\section{The $m$-way Case}

%
%
%
We now generalize to tensors $ \T{C} $ and $ \T{A} $ of any order.

\subsection{The operation for order-$m$ tensors}
For general $m$, we have $ \T{C} \becomes [ \T{A}; \M{X}, \M{X}, \cdots, \M{X} ] $ where $ \T{A} \in
\tensz{m}{n} $, $ \T{C} \in \tensz{m}{p} $, and $ [ \T{A}; \M{X}, \M{X}, \cdots \M{X} ] = \T{A}
\times_\IdxFirst \M{X} \times_\IdxSecond \M{X} \cdots \times_{\IdxLast{m}} \M{X} $.
In our discussion, $ \T{A}$ is a symmetric tensor, as is $\T{C}$ by virtue
of the operation applied to $\T{A}$.

Recall that $ \gamma_{\subfour{j_\IdxFirst}{j_\IdxSecond}{\cdots}{j_{\IdxLast{m}}}} $
denotes the $ (j_\IdxFirst,j_\IdxSecond,\ldots,j_{\IdxLast{m}}) $ element of the
order-$m$ tensor $\T{C}$.
Then, by simple extension of our previous derivations, we find that \[
\begin{array}{l c l}
\gamma_{\subfour{j_\IdxFirst}{j_\IdxSecond}{\cdots}{j_{\IdxLast{m}}}} & = & \T{A} \times_\IdxFirst \Vrow{X}_{j_\IdxFirst} \times_\IdxSecond \Vrow{X}_{j_\IdxSecond} \cdots \times_{\IdxLast{m}} \Vrow{X}_{j_{\IdxLast{m}}} \\
& = & \sum_{i_{\IdxLast{m}}=\IdxFirst}^{\IdxLast{n}} \cdots \sum_{i_\IdxFirst=\IdxFirst}^{\IdxLast{n}} \alpha_{\subfour{i_\IdxFirst}{i_\IdxSecond}{\cdots}{i_{\IdxLast{m}}}} \chi_{\subtwo{j_\IdxFirst}{i_\IdxFirst}} \chi_{\subtwo{j_\IdxSecond}{i_\IdxSecond}} \cdots \chi_{\subtwo{j_{\IdxLast{m}}}{i_{\IdxLast{m}}}}.
\end{array}
\]

\subsection{Simple algorithms for general $m$}

A naive algorithm with a cost of $ (m+1) p^m n^m $ flops is given in
Figure~\ref{fig:3DAlg}~{(bottom left)}.
By comparing the loop structure of the naive algorithms in the 2-way and 3-way
cases, the pattern for a cheaper algorithm (in terms of flops) in the
$m$-way case should become obvious.  Extending the cheaper algorithm in the 3-way
case suggests the algorithm given in Figure~\ref{fig:3DAlg}~(bottom right), This
algorithm requires \[ 2pn^m + 2p^2n^{m-1} + \cdots 2p^{m-1}n =
2\displaystyle\sum_{i=0}^{m-1} p^{i+1}n^{m-i} \mbox{~\rm flops} \] at the expense of requiring workspace for temporary tensors of order $1$ through $m-1$ with modes of dimension $n$.

\subsection{\BCSSshort for general $m$}

We now consider \BCSSshort for the general 
$m$-way case.  The upper hypertriangular
region now contains all indices $
(i_\IdxFirst,i_\IdxSecond,\ldots,i_{\IdxLast{m}}) \text{ where } i_\IdxFirst \le
i_\IdxSecond \le \ldots \le i_{\IdxLast{m}} $.
Using the 3-way case as a guide, one can envision by extrapolation how a
block partitioned order-$m$ tensor looks.
The tensor $\T{A} \in \tensz{m}{n}$ is partitioned into hyper-cubical blocks of
size $\blkDim{\T{A}}^m$.
The blocks lying outside the upper hypertriangular region are not stored.  Once
again, we do not take advantage of symmetry within blocks.

\begin{figure}[tb!]
\centering
\includegraphics[width=3.0in]{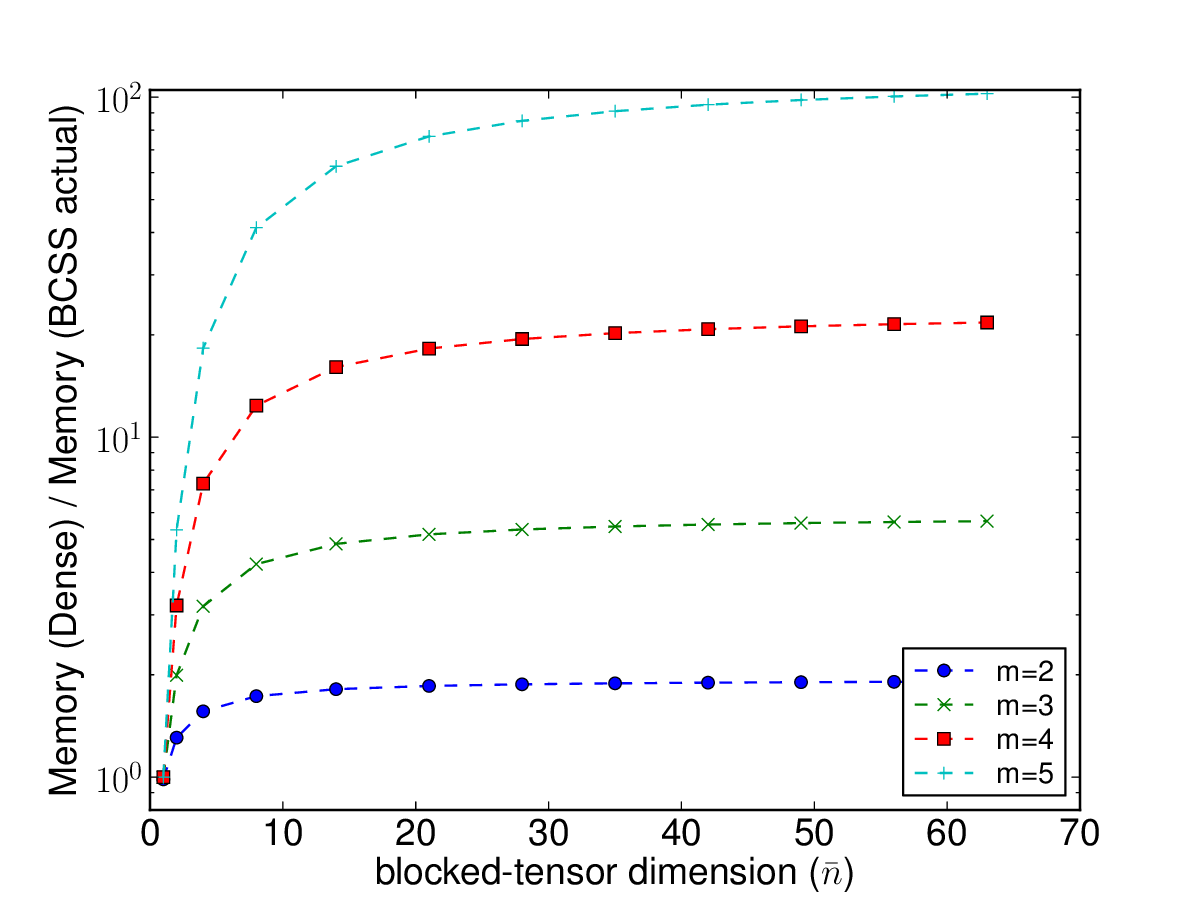}
\caption{Actual storage savings of \BCSSshort on $\T{A}$ with block dimension $\blkDim{\T{A}} = 8$.  This includes the number of entries required for associated meta-data}
\label{fig:approximation_effect_savings}
\end{figure}
 
As summarized in \Tbl{storage_requirements}, storing {\em only} the
upper hypertriangular elements of the tensor $ \T{A} $ requires $ \binom{n + m -
1}{m} $ storage, and \BCSSshort requires $ \binom{\blkedDimA + m -
1}{m}\blkDim{\T{A}}^m $ elements which achieves a savings factor of $m!$ (if
$\blkedDimA$ is large enough).

Although the approximation $ \binom{\blkedDimA + m - 1}{m}\blkDim{\T{A}}^m
\approx \frac{n^m}{m!} $ is used, the lower-order terms have a significant
effect on the actual storage savings factor.  In
\Fig{approximation_effect_savings}, we show the actual storage savings of
\BCSSshort (including storage required for meta-data entries) over storing all
entries of a symmetric tensor.  Examining
Figure~\ref{fig:approximation_effect_savings}, we see that as we increase $m$, a
larger $\blkedDimA$ is required to have the actual storage savings factor
approach the theoretical factor. While this figure only shows the results for a
particular value of $\blkDim{\T{A}}$, the effect applies to all values of
$\blkDim{\T{A}}$.  This idea of blocking has been used in many projects
including the Tensor Contraction Engine (TCE) project~\cite{TCE,
Ragnarsson:2012:BTU:2340215.2340219,EpifanovskyWKLZKMKDK13}.

\subsection{Algorithm-by-blocks for general $m$}
\begin{figure}[tb!]
{%
\setlength{\arraycolsep}{2pt}
\footnotesize
\begin{center}
\begin{tabular}{| l | c | c | c |} \hline
Algorithm & Ops & Total \# of & Temp. \\
& (flops) &  times executed & storage  \\ \hline \hline
{\bf for} $ \blkJ_{\IdxLast{m}} = \IdxFirst, \ldots, \IdxLast{\blkedDimC} $ & & & \\
~~~~~~$
\T{T}^{(m-1)}  
\becomes 
\T{A} \times_{\IdxLast{m}}
\left(\begin{array}{c c c c c}
\M{X}_{\subtwo{\blkJ_{\IdxLast{m}}}{\IdxFirst}} & \cdots & \M{X}_{\subtwo{\blkJ_{\IdxLast{m}}}{\IdxParenLast{\blkedDimA}}} 
\end{array} \right)
$ 
&
$ 2 \blkDim{\T{C}} n^m $ & $ \binom{\blkedDimC}{1} $
& $\blkDim{\T{C}} n^{m-1}$
\\
~~~~~~~~~$\ddots$ & $ \vdots $  & $ \vdots $ & $ \vdots $ \\
 ~~~~~~~~~{\bf for} $ \blkJ_\IdxSecond = \IdxFirst, \ldots , \blkJ_\IdxThird $ & & & \\
~~~~~~~~~~~~~~~$ 
\T{T}^{(1)}
\becomes
\T{T}^{(2)} \times_\IdxSecond
\left(\begin{array}{c c c} 
\M{X}_{\subtwo{\blkJ_\IdxSecond}{\IdxFirst}} & \cdots & \M{X}_{\subtwo{\blkJ_\IdxSecond}{\IdxParenLast{\blkedDimA}}} 
\end{array} \right)
$ 
&
$2 \blkDim{\T{C}}^{m-1} n^2$
&
$
\binom{\blkedDimC + m - 2}{m-1}
$
&
$\blkDim{\T{C}}^{m-1}n$
\\
~~~~~~~~~~~~~~~{\bf for} $ \blkJ_\IdxFirst = \IdxFirst, \ldots , \blkJ_\IdxSecond $ & & &\\
~~~~~~~~~~~~~~~~~~~~~$ 
\T{C}_{\subfour{\blkJ_\IdxFirst}{\blkJ_\IdxSecond}{\cdots}{\blkJ_{\IdxLast{m}}}}
\becomes $ & & & \\
~~~~~~~~~~~~~~~~~~~~~~~~~$ 
\T{T}^{(1)} 
\times_\IdxFirst
\left(\begin{array}{c c c} 
\M{X}_{\subtwo{\blkJ_\IdxFirst}{\IdxFirst}} & \cdots & \M{X}_{\subtwo{\blkJ_\IdxFirst}{\IdxParenLast{\blkedDimA}}} 
\end{array} \right) = $ & & & \\
~~~~~~~~~~~~~~~~~~~~~~~~~\scriptsize$ 
\left(
\begin{array}{c}
\T{T}^{(1)}_\IdxFirst \\
\vdots \\
\T{T}^{(1)}_{\IdxLast{\blkedDimA}}
\end{array}\right)
\times_\IdxFirst
\left(\begin{array}{c c c} 
\M{X}_{\subtwo{\blkJ_\IdxFirst}{\IdxFirst}} & \cdots & \M{X}_{\subtwo{\blkJ_\IdxFirst}{\IdxParenLast{\blkedDimA}}} 
\end{array} \right)
$
&
$2 \blkDim{\T{C}}^{m} n$
&
$
\binom{\blkedDimC + m - 1}{m}
$
&
\\
~~~~~~~~~~~~~~~{\bf endfor} & & & \\
~~~~~~~~~{\bf endfor} & & & \\
~~~~~~~~~$\iddots$ & & & \\
{\bf endfor} & & & \\ \hline 
\multicolumn{4}{c}{}
\\[-0.1in]
\hline
\multicolumn{4}{| l |}{
\begin{minipage}{4.5in}
\[
\begin{array}{l}
\mbox{Total Cost:~~}
\sum_{d=0}^{m-1}
\left( 
2 \blkDim{\T{C}}^{d+1} n^{m-d} 
\binom{\blkedDimC +d}{d+1}
\right)
\mbox{~flops}
\\
\mbox{Total additional storage:~~}
\sum_{d=0}^{m-2}
\left( 
\blkDim{\T{C}}^{d+1} n^{m-1-d} 
\right)
\mbox{~floats}
\end{array}
\]
\end{minipage}}
\\ \hline
\end{tabular}
\end{center}%
}%
\caption{Algorithm-by-blocks for computing $ \T{C} \becomes [ \T{A};
  \M{X}, \cdots , \M{X} ] $.  The algorithm assumes that $ \T{C} $ is
partitioned into blocks of size $ [m,\blkDim{\T{C}}] $, with $ \blkedDimC = \lceil p/ \blkDim{\T{C}}
\rceil $. }
\label{fig:mD}
\end{figure}

For $ \T{C} $ and $ \T{A} $ for general $m$ stored using \BCSSshort, we discuss
how to compute with these blocks.
Assume the partioning discussed above.
Then,
\begin{eqnarray*}
\begin{array}{r @{\hspace{3pt}} l c l}
\T{C}_{\subfour{\blkJ_\IdxFirst}{\blkJ_\IdxSecond}{\cdots}{\blkJ_{\IdxLast{m}}}} 
&=
\sum_{\blkI_\IdxFirst=\IdxFirst}^{\IdxLast{\blkedDimA}}
\cdots
\sum_{\blkI_{\IdxLast{m}}=\IdxFirst}^{\IdxLast{\blkedDimA}} \T{A}_{\subfour{\blkI_\IdxFirst}{\blkI_\IdxSecond}{\cdots}{\blkI_{\IdxLast{m}}}} \times_\IdxFirst
\M{X}_{\subtwo{\blkJ_\IdxFirst}{\blkI_\IdxFirst}} \times_\IdxSecond \M{X}_{\subtwo{\blkJ_\IdxSecond}{\blkI_\IdxSecond}} \cdots \times_{\IdxLast{m}}
\M{X}_{\subtwo{\blkJ_{\IdxLast{m}}}{\blkI_{\IdxLast{m}}}}
\\
&= 
\sum_{\blkI_\IdxFirst=\IdxFirst}^{\IdxLast{\blkedDimA}}
\cdots
\sum_{\blkI_{\IdxLast{m}}=\IdxFirst}^{\IdxLast{\blkedDimA}} [\T{A}_{\subfour{\blkI_\IdxFirst}{\blkI_\IdxSecond}{\cdots}{\blkI_{\IdxLast{m}}}};
\M{X}_{\subtwo{\blkJ_\IdxFirst}{\blkI_\IdxFirst}}, \M{X}_{\subtwo{\blkJ_\IdxSecond}{\blkI_\IdxSecond}}, \cdots, \M{X}_{\subtwo{\blkJ_{\IdxLast{m}}}{\blkI_{\IdxLast{m}}}}]. 
\end{array}
\end{eqnarray*}
This yields the algorithm given in Figure~\ref{fig:mD}, which avoids much
redundant computation, except for within blocks of $\T{C}$ that are symmetric
or partially symmetric.  The algorithm computes temporaries 
\[
\begin{array}{lcl}
\T{T}^{(m-1)} & = & \T{A} \times_{\IdxLast{m}} \M{X}_{\blkJ_{\IdxLast{m}}:} \\
\T{T}^{(m-2)} & = & \T{T}^{(m-1)} \times_{\IdxSecondLast{m}}
\M{X}_{\blkJ_{\IdxSecondLast{m}}:} \\
 & \vdots & \\
\T{T}^{(1)} & = & \T{T}^{(2)} \times_{\IdxSecond}
\M{X}_{\blkJ_{\IdxSecond}:}
\end{array}
\] 
for each index where $\IdxFirst \le \blkJ_{\IdxSecond} \le \blkJ_{\IdxThird} \le
\cdots \le \blkJ_{\IdxLast{m}} < \blkedDimC$.
This algorithm requires $ \blkDim{\T{C}}n^{\IdxLast{m}} +
\blkDim{\T{C}}^2n^{\IdxSecondLast{m}} + \cdots +
\blkDim{\T{C}}^{\IdxSecondLast{m}} n$ extra storage (for
$\T{T}^{(\IdxLast{m})}$ through $\T{T}^{(1)}$, respectively), in addition to the
storage for $ \T{C} $ and $ \T{A} $.

We realize this approach can result in a small loss of symmetry (due to
numerical instability) within blocks.  We do not address this effect at this
time as the asymmetry only becomes a factor when the resulting tensor is used in
subsequent operations.  A post-processing step can be applied to correct
any asymmetry in the resulting tensor.
%
%

\section{Exploiting Partial Symmetry}

We have shown how to reduce the complexity of the {\tt sttsm} operation by
$O(m!)$ in terms of storage.  In this section, we
describe how to achieve the $O((m+1)!/2^m)$ level of reduction in computation.

\subsection{Partial Symmetry}
Recall that in \Fig{mD} we utilized a series of temporaries to compute
the {\tt sttsm} operation.  To perform the computation, we explicitly formed the
temporaries $\T{T}^{(k)}$ and did not take advantage of any symmetry in the
objects' entries.  Because of this, we were only able to see an $O(m)$
reduction in storage and computation.

However, as we now show, there exists partial symmetry within each
temporary that we can exploit to reduce storage and computation as we did for
the output tensor $\T{C}$.  Exploiting this partial symmetry allows the proposed
algorithm to match the theoretical reduction in storage and computation.
\vspace{0.1in}
\begin{theorem}
\label{thm:psymTemp}
Given an order-$m$ tensor $\T{A} \in \R^{I_\IdxFirst \times I_{\IdxSecond} \times \cdots
\times I_{\IdxLast{m}}}$ that has modes $\IdxFirst$ through $\IdxNext{k}$ symmetric
(thus $I_\IdxFirst = I_\IdxSecond = \cdots = I_\IdxNext{k}$), then $\T{C} = \T{A}
\times_{\IdxNext{k}} \M{X}$ has modes $\IdxFirst$ through $\IdxLast{k}$ symmetric.
\end{theorem}

\begin{proof}
We prove this by construction of $\T{C}$.

Since $\T{A}$ has modes $\IdxFirst$ through $\IdxNext{k}$ symmetric, we know (from Section~\ref{sec:mult_sym}) that
\[
\alpha_{\subtwo{\subfour{i_{\IdxFirst}'}{i_{\IdxSecond}'}{\cdots}{i_{\IdxNext{k}}'}}{\subthree{i_{\IdxSecondNext{k}}}{\cdots}{i_{\IdxLast{m}}}}} =
\alpha_{i_{\IdxFirst} i_{\IdxSecond} \cdots i_{\IdxNext{k}} i_{\IdxSecondNext{k}} \cdots i_{\IdxLast{m}}}
\] 
under the relevent permutations. We wish to show that
\[
\gamma_{\subtwo{\subfour{i_{\IdxFirst}'}{i_{\IdxSecond}'}{\cdots}{i_{\IdxLast{k}}'}}{\subthree{j_{\IdxNext{k}}}{\cdots}{j_{\IdxLast{m}}}}} = 
\gamma_{\subtwo{\subfour{i_{\IdxFirst}}{i_{\IdxSecond}}{\cdots}{i_{\IdxLast{k}}}}{\subthree{j_{\IdxNext{k}}}{\cdots}{j_{\IdxLast{m}}}}}
\]
for all indices in $\T{C}$.

\[
\begin{array}{l c l c l}
\gamma_{\subtwo{\subfour{i_{\IdxFirst}'}{i_{\IdxSecond}'}{\cdots}{i_{\IdxLast{k}}'}}{\subthree{j_{\IdxNext{k}}}{\cdots}{j_{\IdxLast{m}}}}} & = & \sum_{\ell=\IdxFirst}^{\IdxLast{n}} \alpha_{\subtwo{\subfour{i_{\IdxFirst}'}{i_{\IdxSecond}'}{\cdots}{i_{\IdxLast{k}}'}}{\subfour{\ell}{j_{\IdxSecondNext{k}}}{\cdots}{j_{\IdxLast{m}}}}}\chi_{j_{\IdxNext{k}}\ell} & & \\
 & = & \sum_{\ell=\IdxFirst}^{\IdxLast{n}} \alpha_{\subtwo{\subfour{i_{\IdxFirst}}{i_{\IdxSecond}}{\cdots}{i_{\IdxLast{k}}}}{\subfour{\ell}{j_{\IdxSecondNext{k}}}{\cdots}{j_{\IdxLast{m}}}}}\chi_{j_{\IdxNext{k}}\ell} & = & \gamma_{\subtwo{\subfour{i_{\IdxFirst}}{i_{\IdxSecond}}{\cdots}{i_{\IdxLast{k}}}}{\subthree{j_{\IdxNext{k}}}{\cdots}{j_{\IdxLast{m}}}}}. \\
\end{array}
\]
Since
$\gamma_{\subtwo{\subfour{i_{\IdxFirst}'}{i_{\IdxSecond}'}{\cdots}{i_{\IdxLast{k}}'}}{\subthree{j_{\IdxNext{k}}}{\cdots}{j_{\IdxLast{m}}}}}
=
\gamma_{\subtwo{\subfour{i_{\IdxFirst}}{i_{\IdxSecond}}{\cdots}{i_{\IdxLast{k}}}}{\subthree{j_{\IdxNext{k}}}{\cdots}{j_{\IdxLast{m}}}}}$
for all indices in $\T{C}$, we can say that $\T{C}$ has modes
$\IdxFirst$ through $\IdxLast{k}$ symmetric.
\end{proof}
 
By applying \Thm{psymTemp} to the algorithm in \Fig{mD}, we observe that all temporaries of the form $\T{T}^{(k)}$ formed have modes
$\IdxFirst$ through $\IdxLast{k}$ symmetric.  It is this
partial symmetry we exploit to further reduce storage and computational
complexity\footnote{It is true that the other modes \emph{may} have symmetry as well, however in general this is not the case, and therefore we do not explore exploiting this symmetry}.

\subsection{Storage}
A generalization of the \BCSSshort scheme can be applied to the
partially symmetric temporary tensors as well.  To do this, we view each
temporary tensor as being comprised of a group of symmetric modes and a group
of non-symmetric modes.  There is once again opportunity for storage savings as
the symmetric indices have redundancies.  As in the \BCSSshort case for
symmetric tensors, unique blocks are stored and meta-data indicating how to
transform stored blocks to the corresponding block is stored for all redundant
block entries.

\subsection{Computation}
Recall that each temporary is computed via \newline $\T{T}^{(k)} = \T{T}^{(k+1)}
\times_{k+1} \M{B}$ where $\T{T}^{(k)}$ and $\T{T}^{(k+1)}$ have associated
symmetries ($\T{T}^{\IdxNext{(m)}} =
\T{A}$ when computing $\T{T}^{(\IdxLast{m})}$), and $\M{B}$ is some matrix.  We can rewrite this operation as:

\[
\begin{array}{l c l}
\T{T}^{(k)} & = & \T{T}^{(k+1)} \times_{\IdxSecondNext{k}} \M{B} = \T{T}^{(k+1)} \times_{\IdxFirst} \M{I} \times_{\IdxSecond} \cdots \times_{\IdxNext{k}} \M{I} \times_{\IdxSecondNext{k}} \M{B} \times_{\IdxThirdNext{k}} \M{I} \times_{\IdxFourthNext{k}} \cdots \times_{\IdxLast{m}} \M{I}, \\ 
 \end{array}
 \]
 
 \noindent where $\M{I}$ is the first $p$ rows of the $n \times n$ identity
 matrix.
 An algorithm akin to that of \Fig{mD} can be created (care is taken to only update unique output blocks) to
 perform the necessary computation.  Of course, computing with the identity
 matrix is wasteful, and therefore, we only implicitly compute with the identity
 matrix to save on computation.
 
 \newpage
 \subsection{Analysis}
\begin{table}[ht!]
\scriptsize
\centering
\begin{tabular}{| c | c | c |}
\hline
 & \BCSSshort & \Dense \\ \hline
 \begin{tabular}{@{}c@{}}$\T{A}$ \\ (elements)\end{tabular} & $\blkDim{\T{A}}^m\binom{\blkedDimA + m - 1}{m} + \blkedDimA^m$ & $n^m$ \\\hline
 \begin{tabular}{@{}c@{}}$\T{C}$ \\ (elements)\end{tabular} & $\blkDim{\T{C}}^m\binom{\blkedDimC + m - 1}{m} + \blkedDimC^m$ & $p^m$ \\\hline
 \begin{tabular}{@{}c@{}}$\M{X}$ \\ (elements)\end{tabular} & $pn$ & $pn$ \\\hline
 \begin{tabular}{@{}c@{}}All temporaries \\ (elements)\end{tabular} & $\blkDim{\T{C}}\blkDim{\T{A}}^{m-1}\sum_{d=0}^{m-2}\left(\binom{\blkedDimA + m - d - 2}{m - d - 1}\left(\frac{\blkDim{\T{C}}}{\blkDim{\T{A}}}\right)^{d} + \blkedDimA^{d+1} \right)$ & $pn^{m-1}\sum_{d=0}^{m-2}\left(\frac{p}{n}\right)^{d}$ \\\hline
 \begin{tabular}{@{}c@{}}Computation \\ (flops)\end{tabular} & $2\blkedDimA\blkDim{\T{C}}\blkDim{\T{A}}^{m}\sum_{d=0}^{m-1}\binom{\blkedDimC+d}{d+1}\binom{\blkedDimA + m - d - 2}{m - d - 1}\left(\frac{\blkDim{\T{C}}}{\blkDim{\T{A}}}\right)^{d}$ & $2pn^m\sum_{d=0}^{m-1}\left(\frac{p}{n}\right)^{d}$ \\\hline
 \begin{tabular}{@{}c@{}}Permutation \\ (memops)\end{tabular} & $\left(\blkedDimA + 2\frac{\blkDim{\T{C}}}{\blkDim{\T{A}}}\right)\blkDim{\T{A}}^m\sum_{d=0}^{m-1}\binom{\blkedDimC+d}{d+1}\binom{\blkedDimA + m - d - 2}{m - d - 1}\left(\frac{\blkDim{\T{C}}}{\blkDim{\T{A}}}\right)^{d}$ & $\left(1+2\frac{p}{n}\right)n^m\sum_{d=0}^{m-1}\left(\frac{p}{n}\right)^d$\\\hline 
\end{tabular}

\caption{Costs associated with different algorithms for computing $\T{C} = [\T{A}; \M{X}, \ldots, \M{X}]$.  The \BCSSshort column takes advantage of partial-symmetry within the temporaries.  The $\blkedDimA^{m}$, $\blkedDimC^{m}$, and $\blkedDimA^{d+1}$ terms correspond to the number of meta-data elements associated with our choice of storage scheme.}
\label{tbl:MinCosts}
\end{table}
 
 Utilizing this optimization allows us to arrive at the final cost functions for
 storage and computation shown in \Tbl{MinCosts}.
 
 Taking, for example, the computational cost (assuming $n=p$ and $\blkDim{\T{A}}
 = \blkDim{\T{C}}$), we have the following expression for the cost of the \BCSSshort
 algorithm:
 
 \[
 \begin{array}{lcl}
 \lefteqn{2\blkedDimA\blkDim{\T{C}}\blkDim{\T{A}}^{m}\sum_{d=0}^{m-1}\binom{\blkedDimC+d}{d+1}\binom{\blkedDimA + m - d - 2}{m - d - 1}\left(\frac{\blkDim{\T{C}}}{\blkDim{\T{A}}}\right)^{d}} \\
 & = & 2n\blkDim{\T{A}}^{m}\sum_{d=0}^{m-1}\binom{\blkedDimA+d}{d+1}\binom{\blkedDimA + m - (d + 1) -1}{m - (d + 1)}\left(1\right)^{d} \\
  & & \\
 & \approx & 2n\blkDim{\T{A}}^{m}\binom{2\blkedDimA + m - 2}{m} \approx 2n\blkDim{\T{A}}^{m}\frac{\left(2\blkedDimA\right)^m}{m!} = \frac{\left(2n\right)^{m+1}}{m!}. \\
 \end{array}
 \]
 To achieve this approximation, the Vandermonde identity, which states that
\[
\binom{m + n}{r} = \sum_{k=0}^{r}\binom{m}{k}\binom{n}{r-k}
\]
was employed.  
\begin{table}[tb!]
\centering
\begin{tabular}{| c | c | c |}
\hline
 & \BCSSshort & \Dense \\ \hline
 \begin{tabular}{@{}c@{}}$\T{A}$, $\T{C}$ \\ (elements)\end{tabular} & $\blkDim{\T{A}}^m\binom{\blkedDimA + m - 1}{m}$ & $n^m$ \\\hline

 \begin{tabular}{@{}c@{}}$\M{X}$ \\ (elements)\end{tabular} & $n^2$ & $n^2$ \\\hline
 \begin{tabular}{@{}c@{}}All temporaries \\ (elements)\end{tabular} & $\frac{n^m}{m!}$ & $(m-1)n^m$ \\\hline
 \begin{tabular}{@{}c@{}}Computation \\ (flops)\end{tabular} & $\frac{(2n)^{m+1}}{m!}$ & $2mn^{m+1}$ \\\hline
 \begin{tabular}{@{}c@{}}Permutation \\ (memops)\end{tabular} & $(\blkedDimA + 2)\frac{(2n)^{m}}{m!}$ & $3mn^m$\\\hline 
\end{tabular}

\caption{Approximate costs associated with different algorithms for computing $\T{C} = [\T{A}; \M{X}, \ldots, \M{X}]$.  The \BCSSshort column takes advantage of partial symmetry within the temporaries.  In the above costs, it is assumed that the tensor dimensions and block dimensions of $\T{A}$ and $\T{C}$ are equal, i.e. $n=p$ and $\blkDim{\T{A}} = \blkDim{\T{C}}$.  We assume $n^m >> \blkedDimA^m$.}
\label{tbl:LimitMinCosts}
\end{table}
Using similar approximations, we arrive at the estimates summarized in
 \Tbl{LimitMinCosts}.
 
Comparing this computational cost to that of the \Dense algorithm (as given in \Tbl{LimitMinCosts}), we see that the \BCSSshort algorithm achieves a reduction of 

\[
\frac{\text{\Dense cost}}{\text{\BCSSshort cost}} = \frac{2mn^{m+1}}{\left(\frac{(2n)^{m+1}}{m!}\right)} \approx \frac{(m+1)!}{2^m} \text{ \fromtoMartin{flops}{}} 
\]
in terms of computation.

 \subsection{Analysis relative to minimum}
 As we are storing some elements redundantly, it is important to compare how the
 algorithm performs compared to the case where we store no extra elements, that
 is, we only compute the unique entries of $\T{C}$.  Assuming $\T{A} \in
 \tensz{m}{n}$, $\T{C} \in \tensz{m}{p}$, $p=n$, and $\blkDim{\T{A}} =
 \blkDim{\T{C}}=1$, the cost of computing the {\tt sttsm} operation is
 
 \[
 2n\sum_{d=0}^{m-1}\binom{n + d}{d + 1}\binom{n + m - d - 2}{m - d - 1} \approx 2n\binom{2n+m-2}{m} \approx 2n\frac{(2n)^{m}}{m!} = \frac{(2n)^{m+1}}{m!}, 
 \]
which is of the same order as our blocked algorithm.

\subsection{Summary}
 
\begin{figure}[tbp!]
\centering
\begin{tabular}{@{\hspace{-.05in}} c @{\hspace{-0.1in}} c @{\hspace{-0.1in}} c}
\includegraphics[width=1.8in]{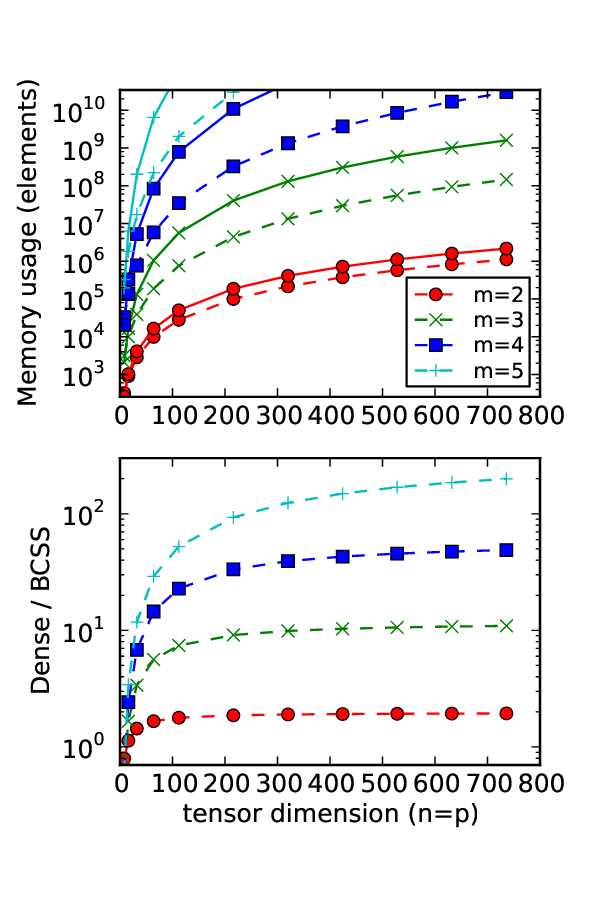} & 
\includegraphics[width=1.8in]{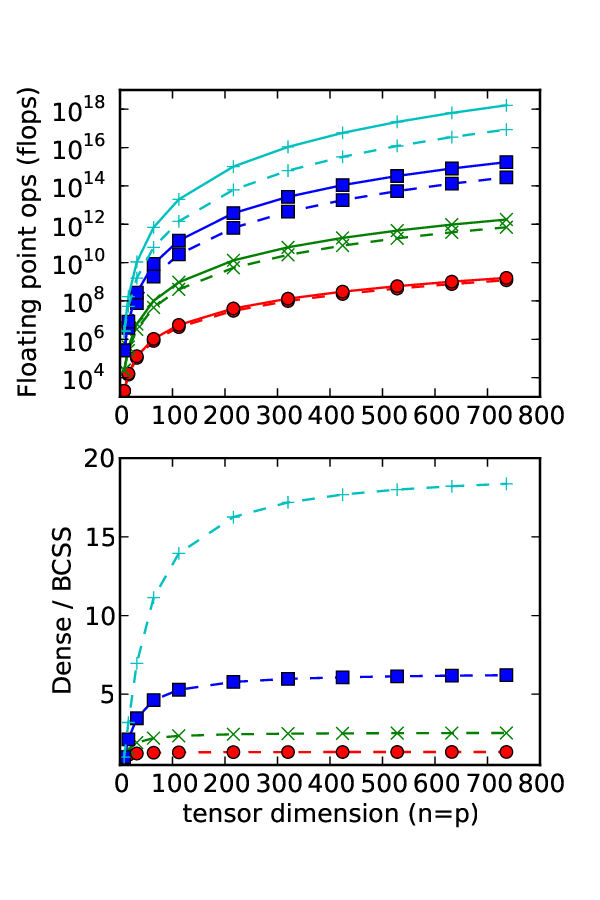} & 
\includegraphics[width=1.8in]{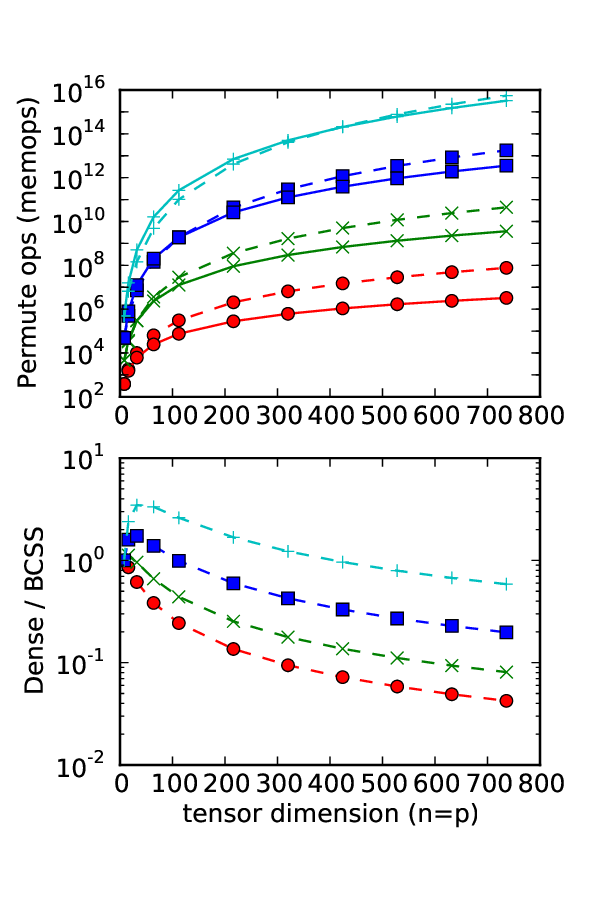} \\
\end{tabular}

\vspace{-0.2in}
 
\caption{Comparison of \Dense to \BCSSshort algorithms for fixed block size. Solid and dashed lines correspond to \Dense and \BCSSshort, respectively. From left to right: 
storage requirements, cost from computation (flops), 
cost from permutations (memops).  For these graphs $\blkDim{\T{A}}=\blkDim{\T{C}}=8$.}
\label{fig:dense_compact_comparison}

\vspace{0.1in}

\centering
\begin{tabular}{@{\hspace{-.05in}} c @{\hspace{-0.1in}} c @{\hspace{-0.1in}} c}
\includegraphics[width=1.8in]{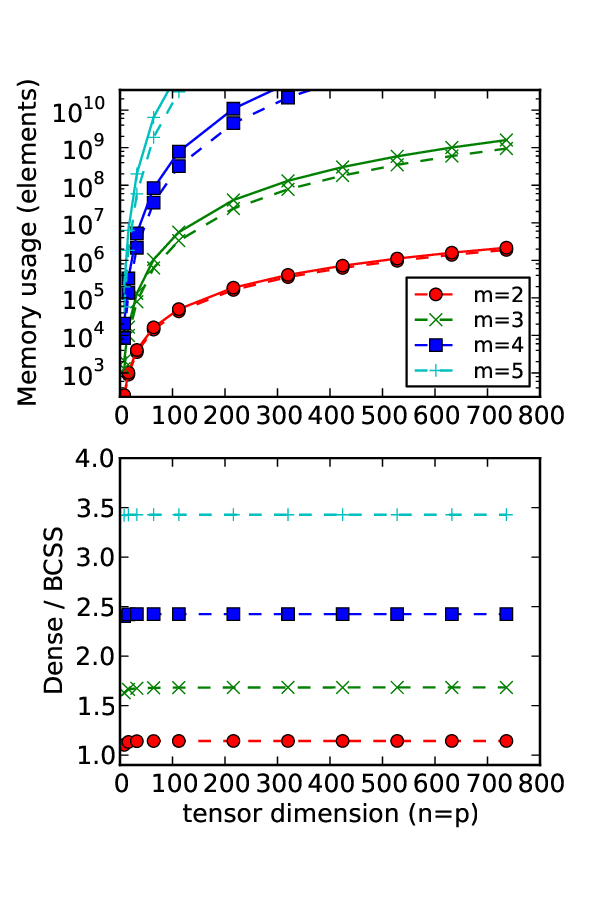} & 
\includegraphics[width=1.8in]{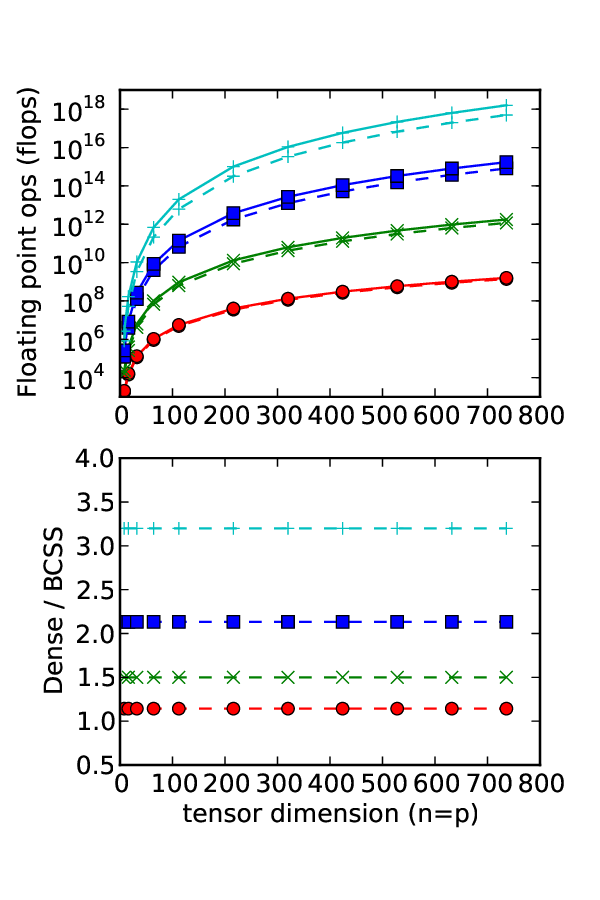} & 
\includegraphics[width=1.8in]{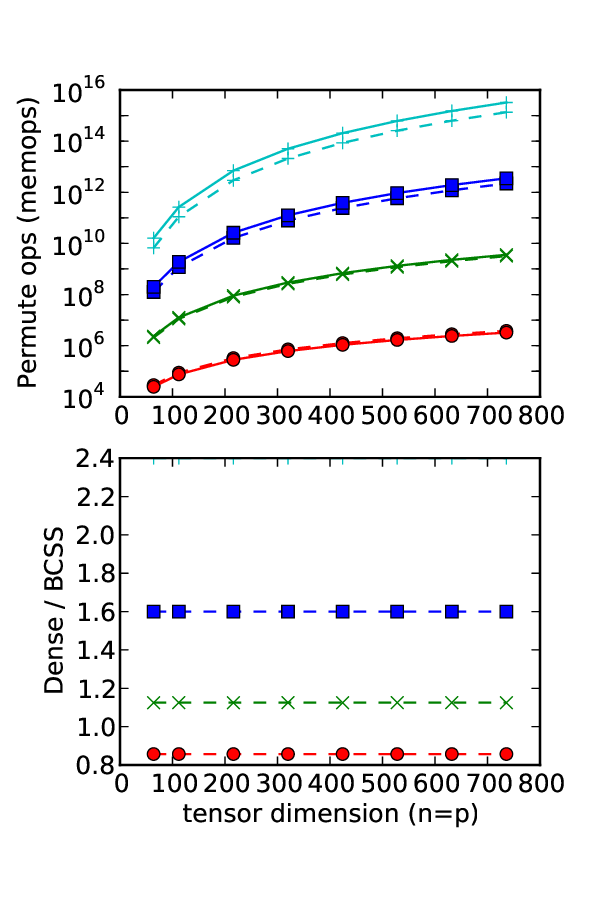} \\
\end{tabular} 

\vspace{-0.2in}

\caption{Comparison of \Dense to \BCSSshort algorithms for fixed number of blocks. Solid and dashed lines correspond to \Dense and \BCSSshort, respectively. From left to right: 
storage requirements, cost from computation (flops), 
cost from permutations (memops).  Here $\blkedDimA=\blkedDimC=2$.}
\label{fig:dense_compact_comparison_p2}
\end{figure}

Figures~\ref{fig:dense_compact_comparison}--\ref{fig:dense_compact_comparison_p2}
illustrate the insights discussed in this section.
The (exact) formulae developed for storage, flops, and memops are used to
compare and contrast dense storage with \BCSSshort.

In Figure~\ref{fig:dense_compact_comparison}, the top graphs report storage, flops, and memops (due to permutations) 
as a function of tensor dimension ($n$), for different tensor orders ($m$), 
for the case where the storage block size is relatively small
($b_{\T{A}} = b_{\T{C}} = 8 $).  The bottom graphs report the same 
information, but as a ratio.
 The graphs illustrate that \BCSSshort dramatically reduces the
 required storage and the proposed algorithms reduce the flops requirements for the {\tt sttsm}
operation, at the expense of additional memops due to the encountered
permutations.
 
In Figure~\ref{fig:dense_compact_comparison_p2}, a similar analysis is
given, but for the case where the block size is half the tensor
dimension (i.e., $ \bar n = \bar p = 2 $).  It shows that the memops can be
greatly reduced by increasing the storage block dimensions, but this then
adversely affects the storage and computational benefits.

It would be tempting to discuss how to choose an optimal block dimension. 
However, the real issue is that the overhead of permutation should be reduced
and/or eliminated.  Once that is achieved, in future research, the question of
how to choose the block size becomes relevant.

\section{Experimental Results}
In this section, we report on the performance attained by an implementation of
the discussed approach.
It is important to keep in mind that the current state of the art of tensor
computations is first and foremost concerned with reducing memory requirements
so that reasonably large problems can be executed.  This is where taking
advantage of symmetry is important.  With that said, another primary concern
 is ensuring the overall time of computation is reduced.  To achieve this, a
 reduction in the number of floating point operations as well as an
 implementation that computes the necessary operations efficiently are both desired.
Second to that is the desire to reduce the number of floating
point operations to the minimum required. The provided analysis shows that our
algorithms perform the minimum number of floating point operations (under approximation).  Although
our algorithms do not yet perform these operations efficiently, our results show
that we are still able to reduce the computation time (in some cases
significantly).

\subsection{Target architecture}

We report on experiments on a single core of a Dell PowerEdge R900 server
consisting of four six-core Intel Xeon 7400 processors and 96 GBytes of memory.
Performance experiments were gathered under the GNU/Linux 2.6.18 operating
system. Source code was compiled by the GNU C compiler, version 4.1.2. All
experiments were performed in double-precision floating-point arithmetic on
randomized real domain matrices and tensors.
The implementations were linked to the OpenBLAS 0.1.1
library~\cite{OpenBLAS,OpenBLASpaper}, a fork of the GotoBLAS2 implementation of
the BLAS~\cite{Goto:2008:AHP,1377607}.
As noted, most time is spent in the permutations necessary to cast computation
in terms of the BLAS matrix-matrix multiplication routine {\tt dgemm}.  Thus,
the peak performance of the processor and the details of the performance
attained by the BLAS library are mostly irrelevant at this stage.
The experiments merely show that the new approach to storing matrices as well as
the algorithm that takes advantage of symmetry has promise, rather than making a
statement about optimality of the implementation.
For instance, as argued previously, we know that tensor permutations can
dominate the time spent computing the {\tt sttsm} operation.
These experiments make no attempt to reduce the number of tensor permutations
required when computing a single block of the output.  Algorithms reducing the
effect of tensor permutations have been specialized for certain tensor
operations and have been shown to greatly increase the performance of routines
using them~\cite{Phan:2012:ComputationCP,ToBr09}. Much room for
improvement remains.

\subsection{Implementation}

The implementation was coded in a style inspired by the {\tt libflame}
library~\cite{libflame_ref,CiSE09} \fromtoMartin{}{and can be found in the
Google Code {\tt tlash} project (\url{code.google.com/p/tlash})}.  An API similar
to the FLASH API~\cite{FLAWN12} for storing matrices as matrices of blocks and
implementing algorithm-by-blocks was defined and implemented.
Computations with the (tensor and matrix) blocks were implemented as the
discussed sequence of permutations interleaved with calls to the {\tt dgemm}
BLAS kernel.  No attempt was yet made to optimize these permutations.  However,
an apples-to-apples comparison resulted from using the same sequence of
permutations and calls to {\tt dgemm} for both the experiments that take
advantage of symmetry and those that store and compute with the tensor densely,
ignoring symmetry.

\begin{figure}[t!]
\begin{tabular}{@{} c @{\hspace{-0.85in}} c @{\hspace{-0.85in}} c}
\includegraphics[width=2.5in]{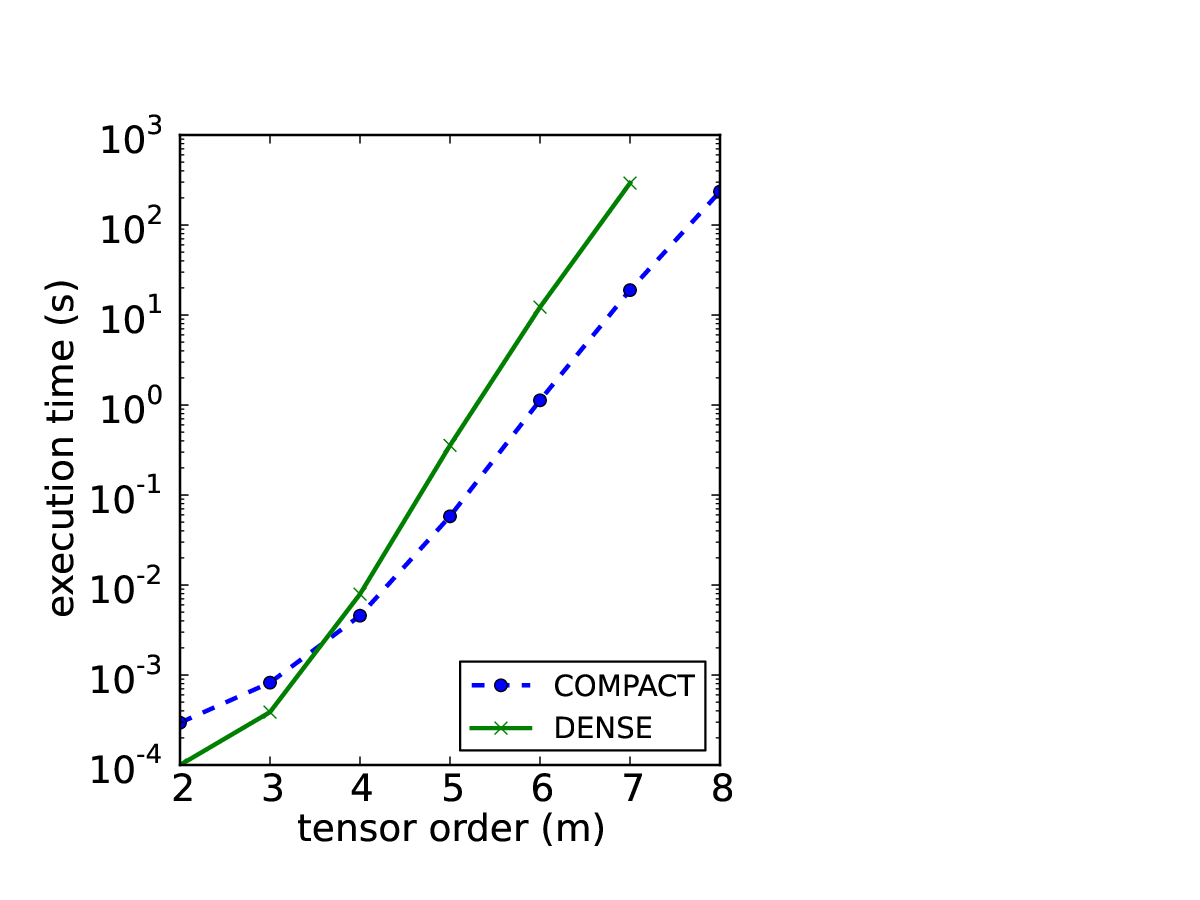} 
& \includegraphics[width=2.5in]{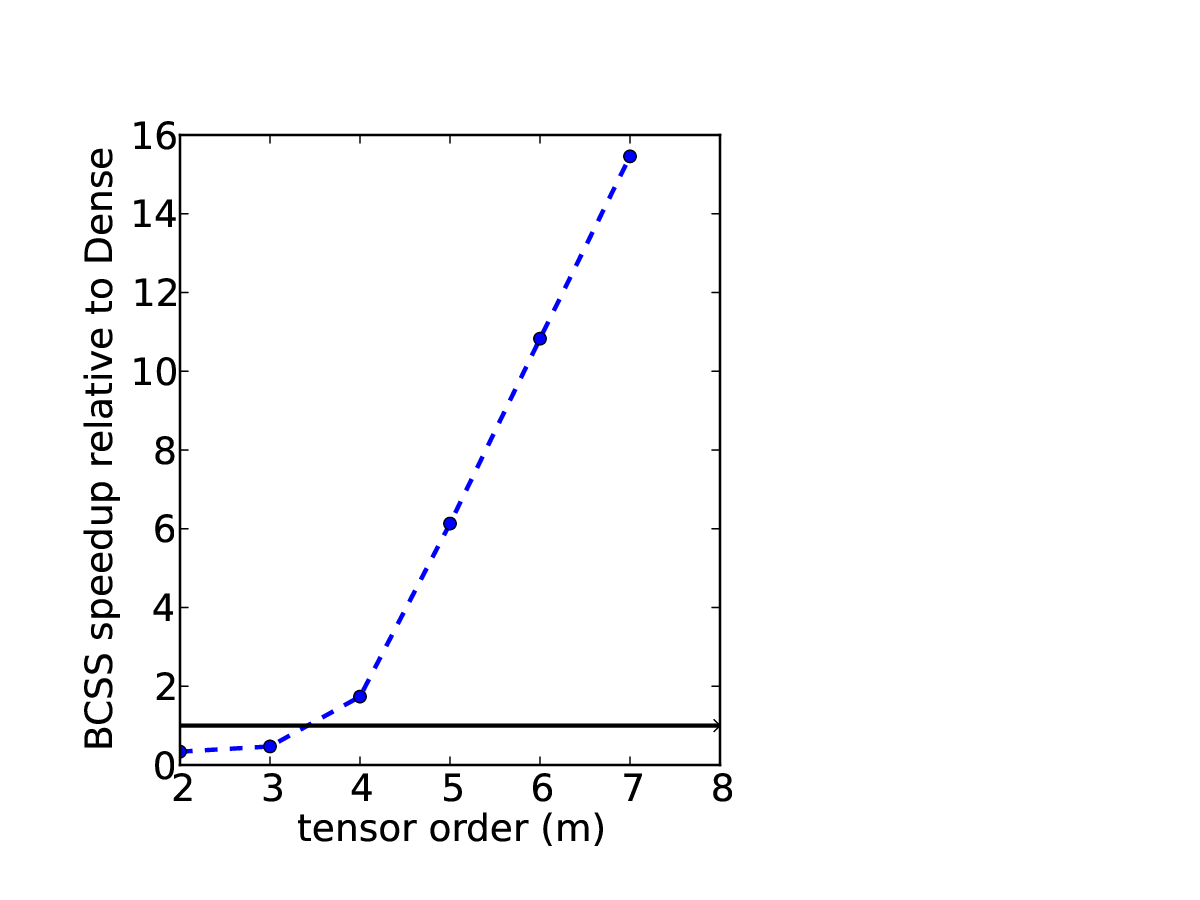} 
& 
\includegraphics[width=2.5in]{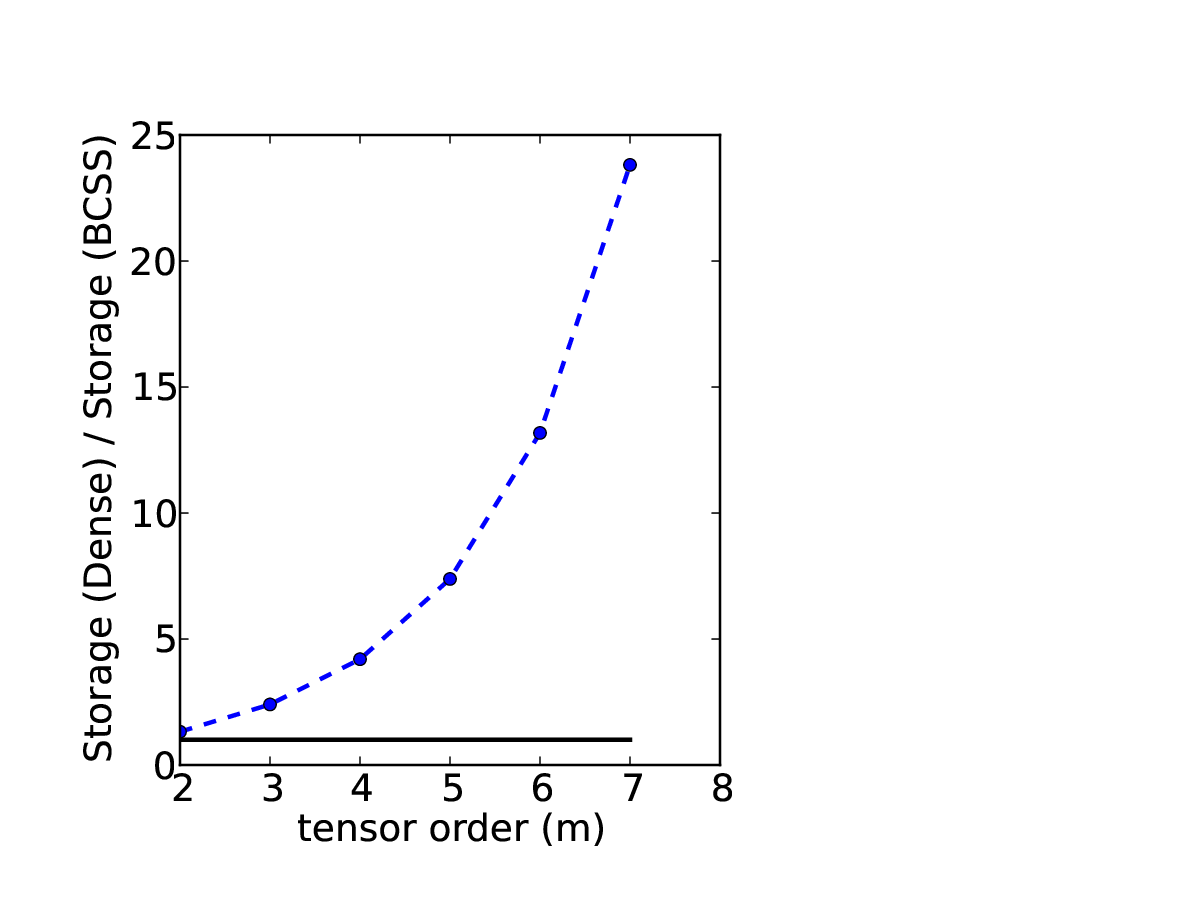}
\end{tabular}
\caption{Experimental results when $ n = p = 16 $, $ b_{\T{A}} = b_{\T{C}} = 8 $ and 
the tensor order, $ m $, is varied.  For $ m = 8 $, storing $ \T{A} $ and $ \T{C} $ without 
taking advantage of symmetry requires too much memory.  The solid black line is used to indicate a unit ratio.}
\label{fig:varyMTests}

%
%
%
%
%
%
%
%
%
%
\end{figure}
\begin{figure}[t!]
\begin{tabular}{@{} c @{\hspace{-0.85in}} c @{\hspace{-0.85in}} c}
\includegraphics[width=2.5in]{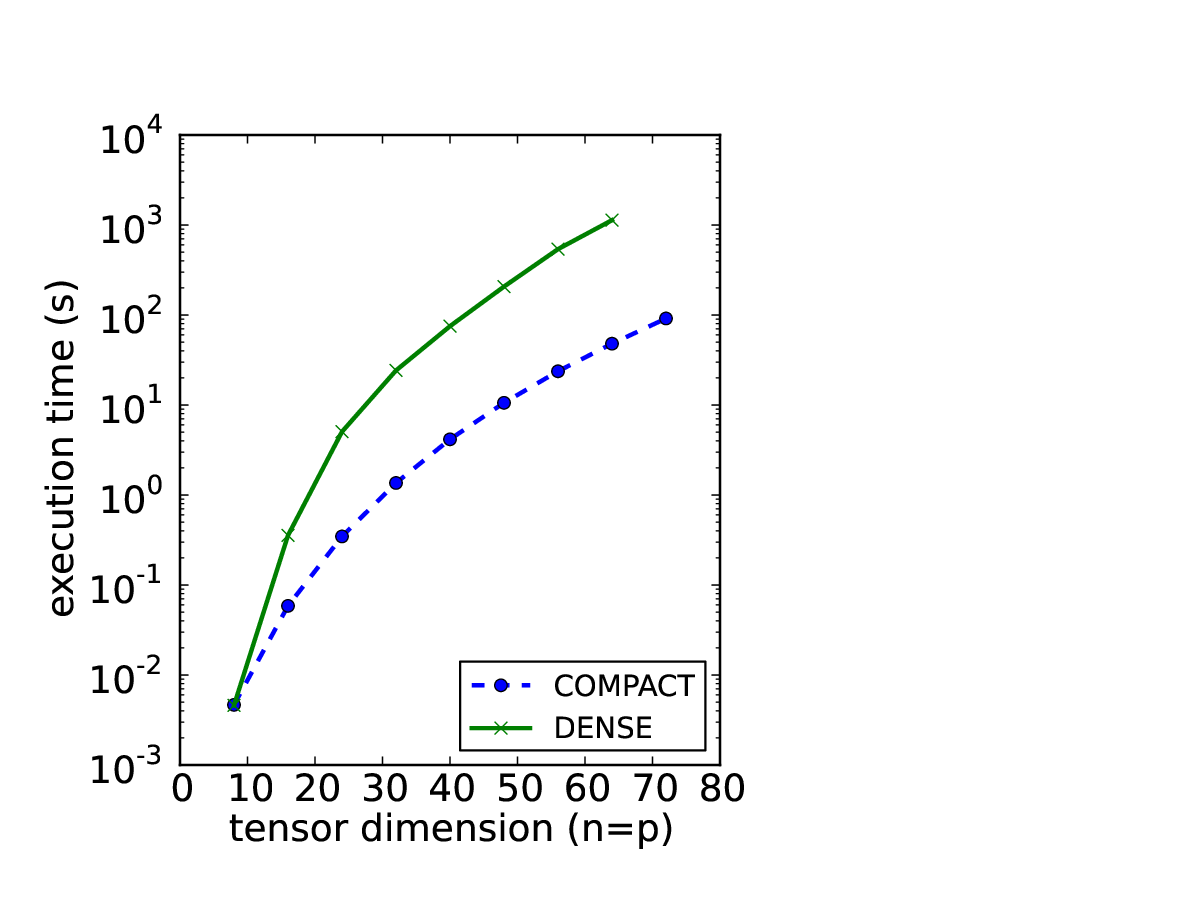} 
& \includegraphics[width=2.5in]{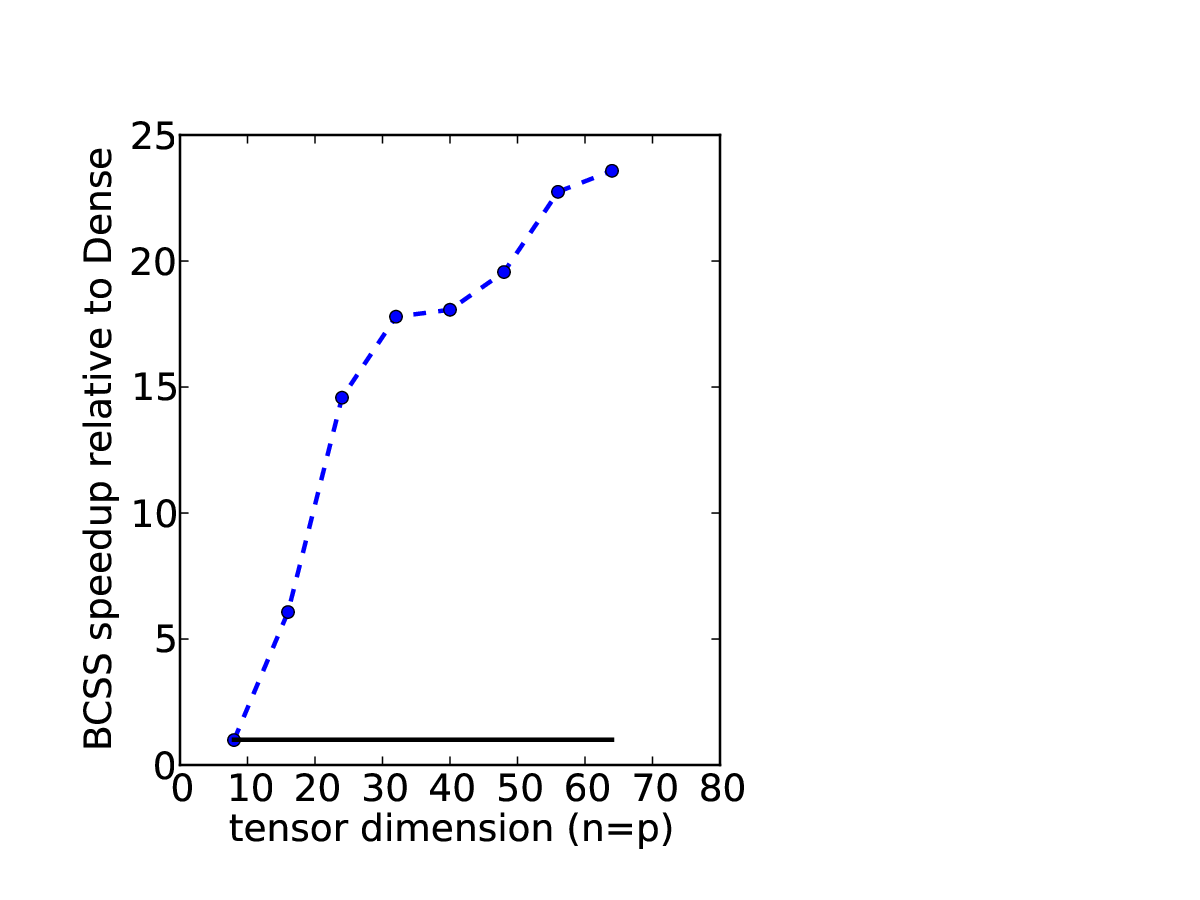} 
&
\includegraphics[width=2.5in]{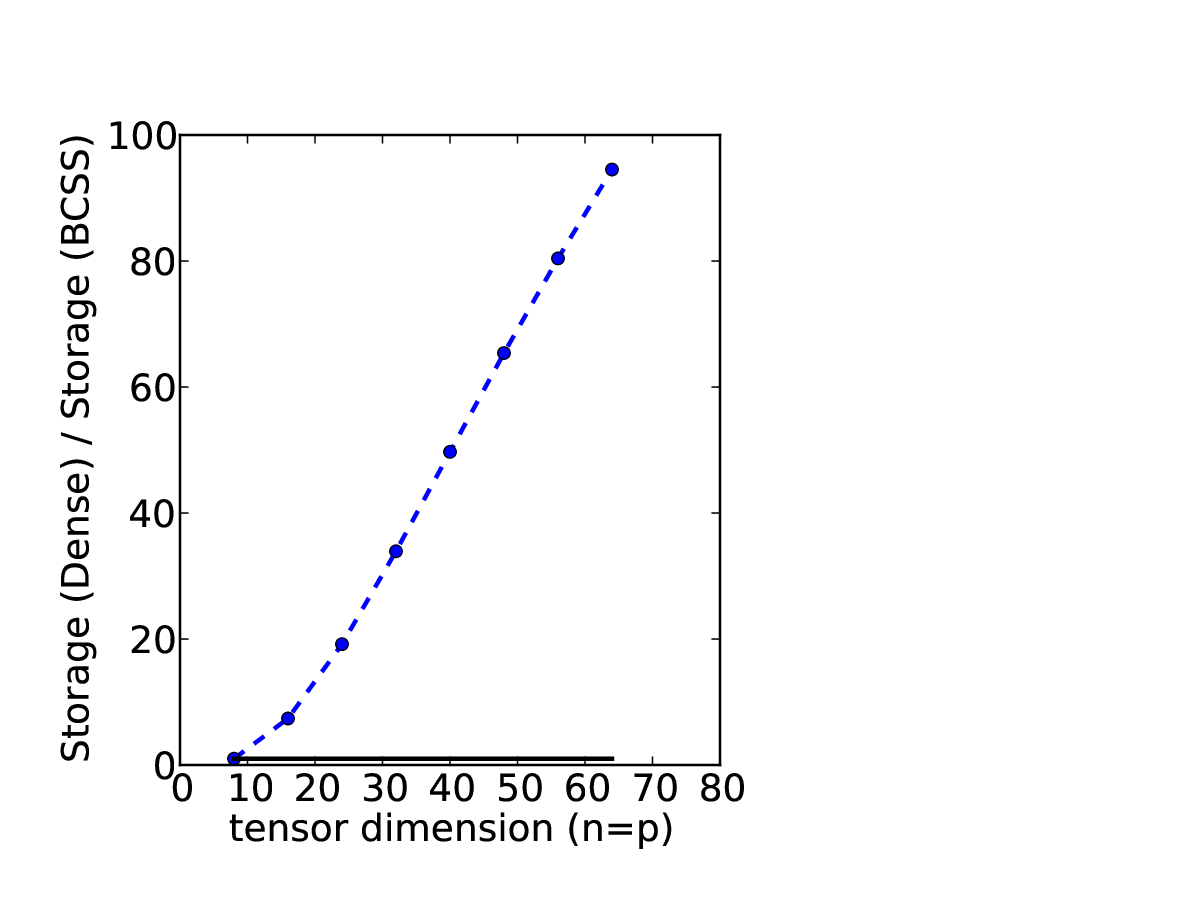}
\end{tabular}
\caption{Experimental results when the order $ m = 5 $, $ b_{\T{A}} = b_{\T{C}} = 8 $ and 
the tensor dimensions $ n = p $ are varied.  For $n=p=72$, storing $\T{A}$ and $\T{C}$ without taking advantage of symmetry requires too much memory.  The solid black line is used to indicate a unit ratio.}
\label{fig:varyNTests}
\vspace{0.3in}
\begin{tabular}{@{} c @{\hspace{-0.85in}} c @{\hspace{-0.85in}} c}
\includegraphics[width=2.5in]{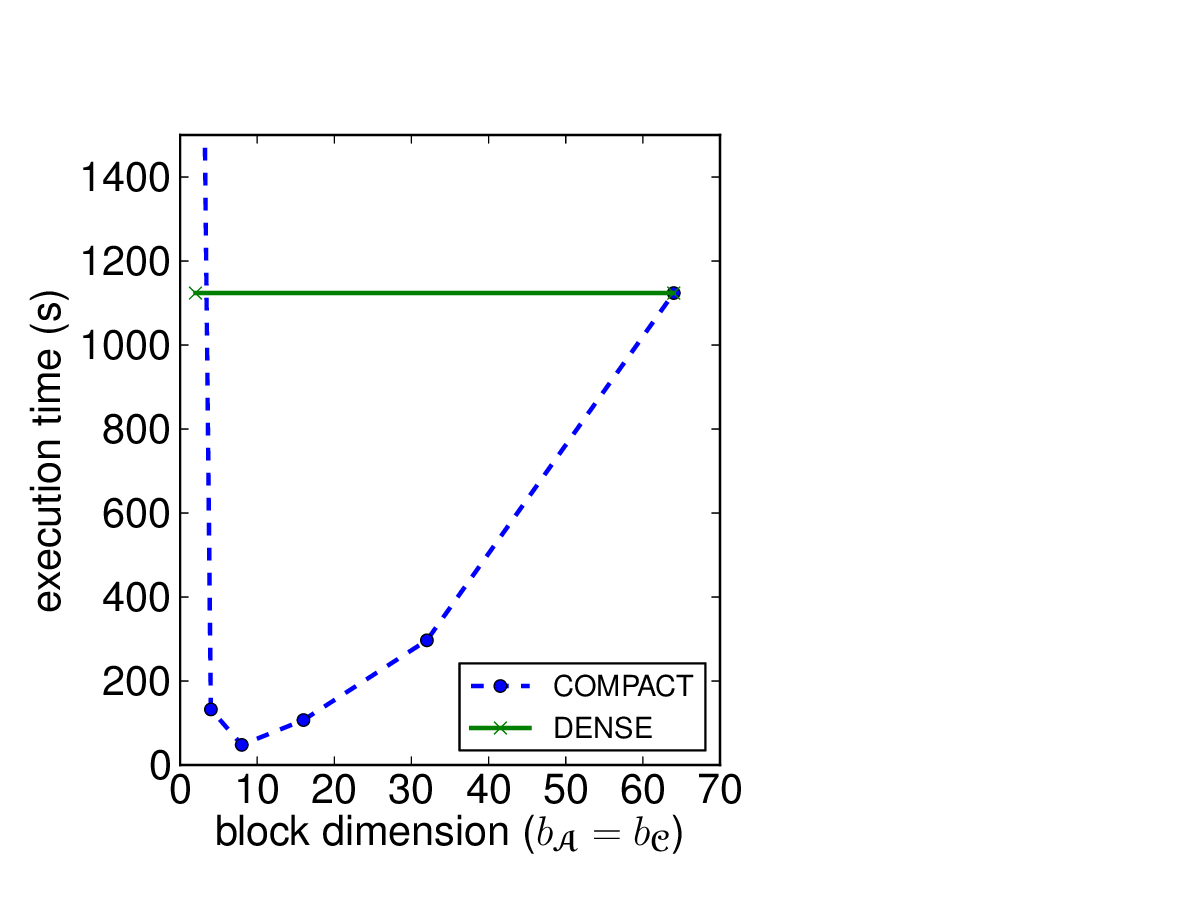} 
& \includegraphics[width=2.5in]{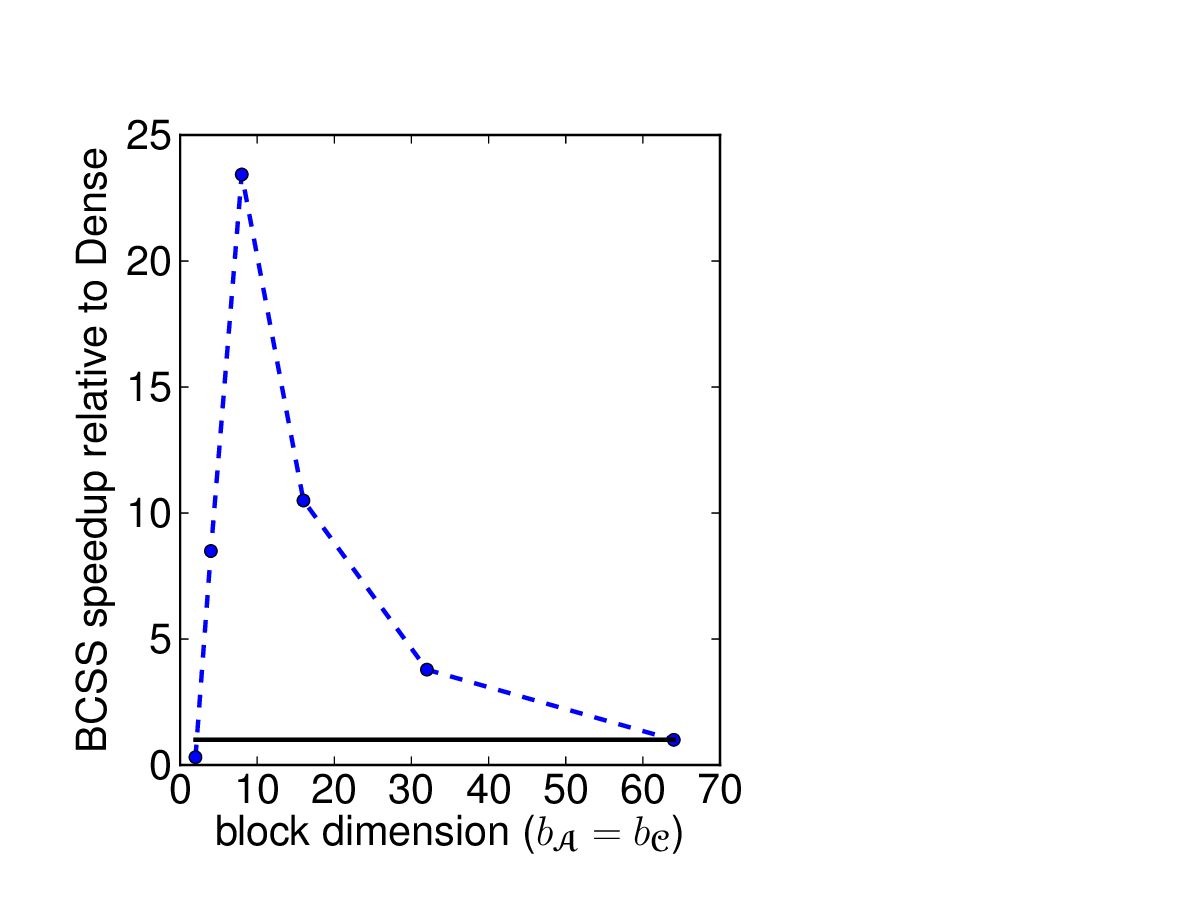} &
\includegraphics[width=2.5in]{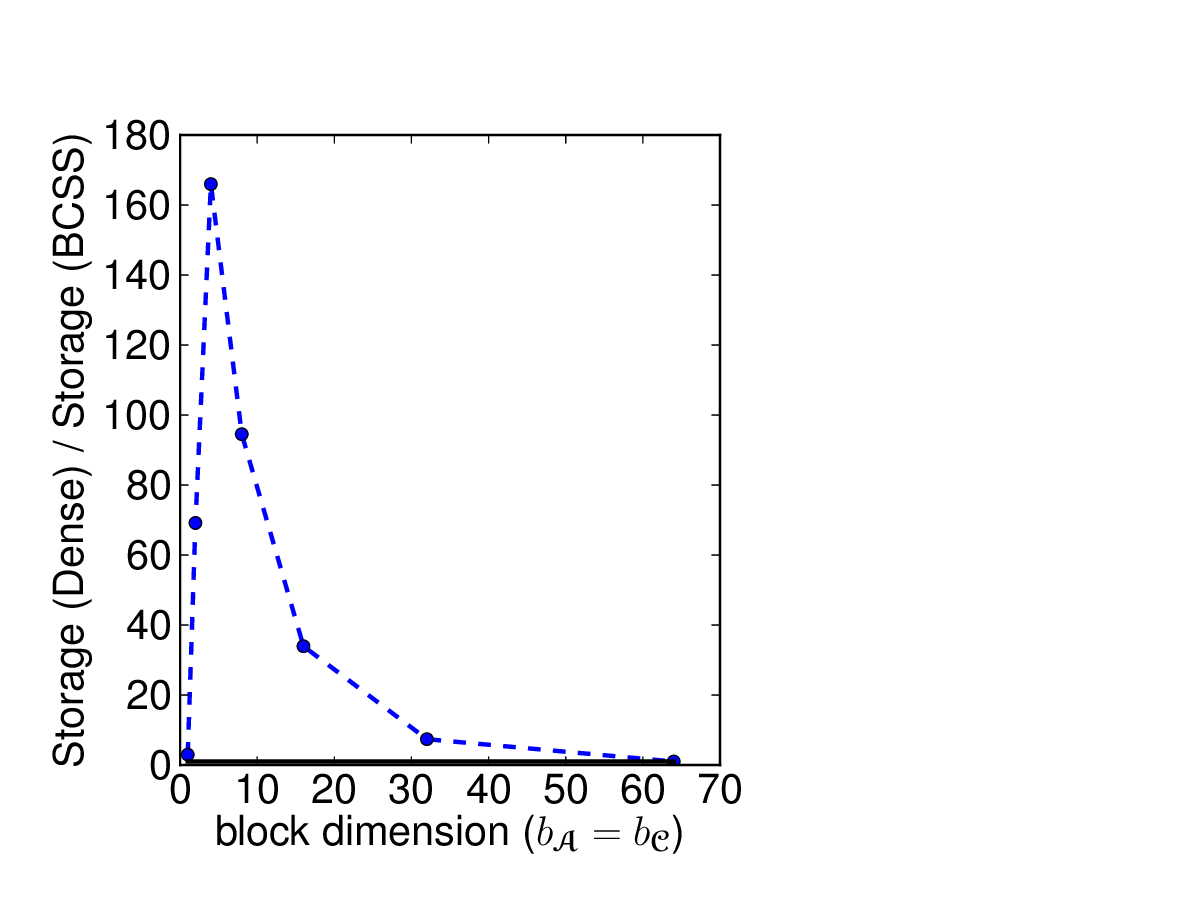}
\end{tabular}
\caption{Experimental results when the order $ m = 5 $, $ n = p = 64 $ and 
the block dimensions $ b_{\T{A}} = b_{\T{C}} $ are varied.  The solid black line is used to indicate a unit ratio.}
\label{fig:varyBTests}
\vspace{0.3in}
\end{figure}
\newpage

\begin{figure}[t!]
\begin{tabular}{@{} c @{\hspace{-0.85in}} c @{\hspace{-0.85in}} c}
\includegraphics[width=2.5in]{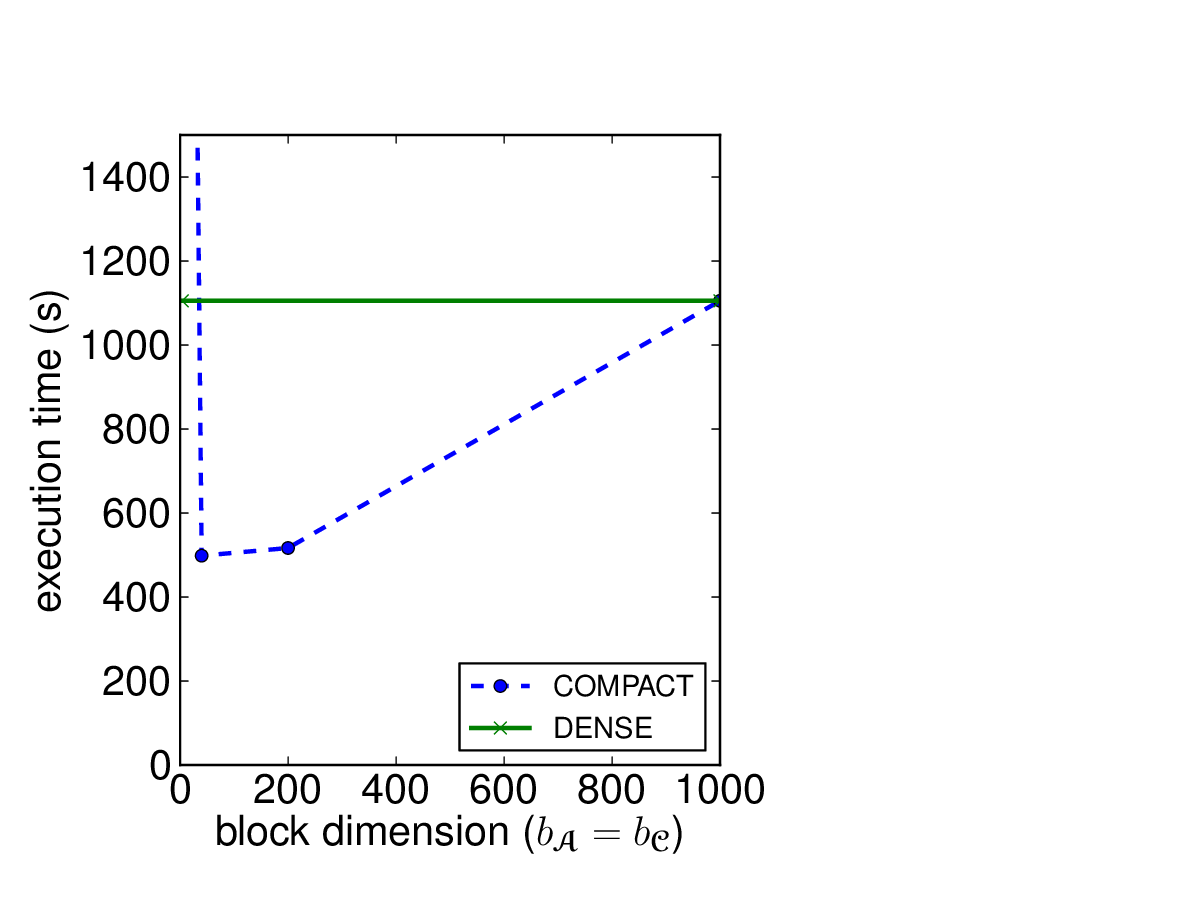} 
& \includegraphics[width=2.5in]{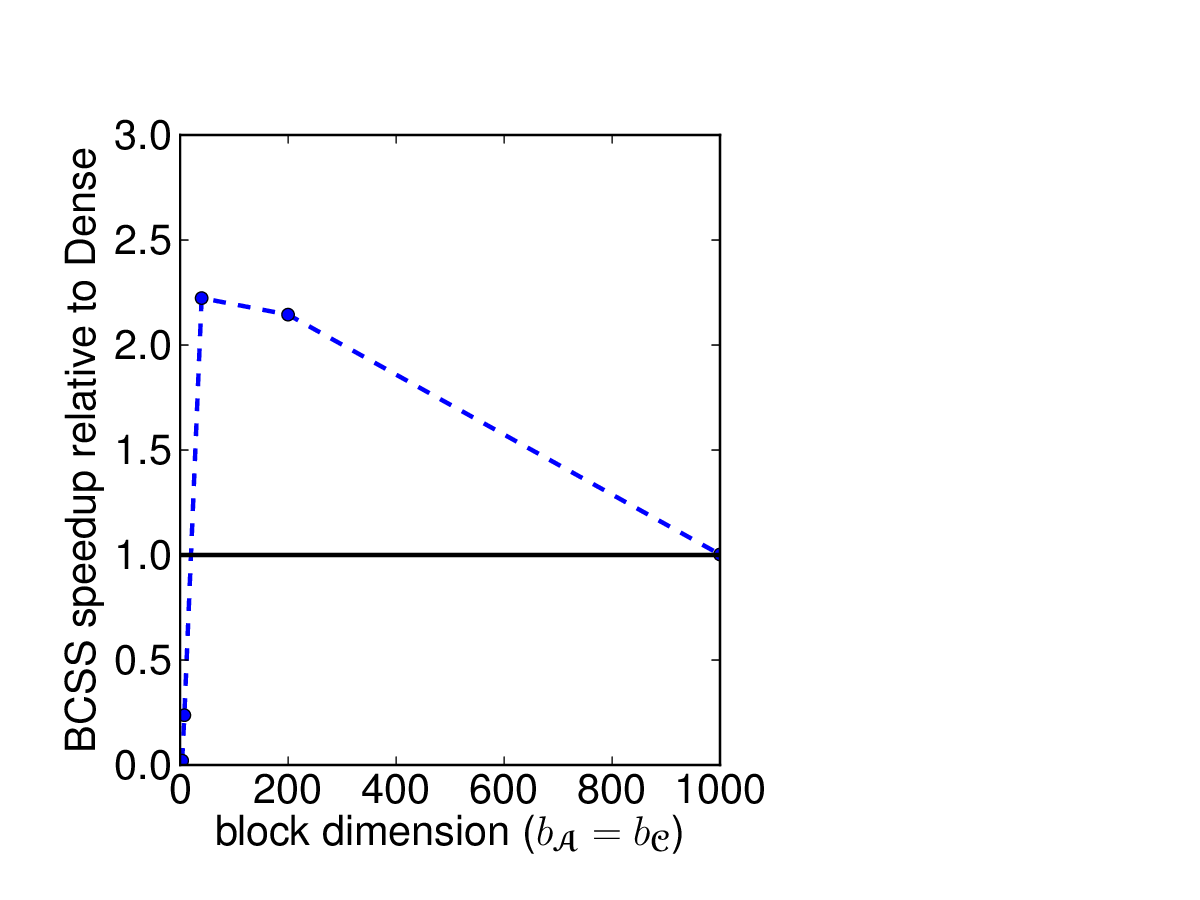} &
\includegraphics[width=2.5in]{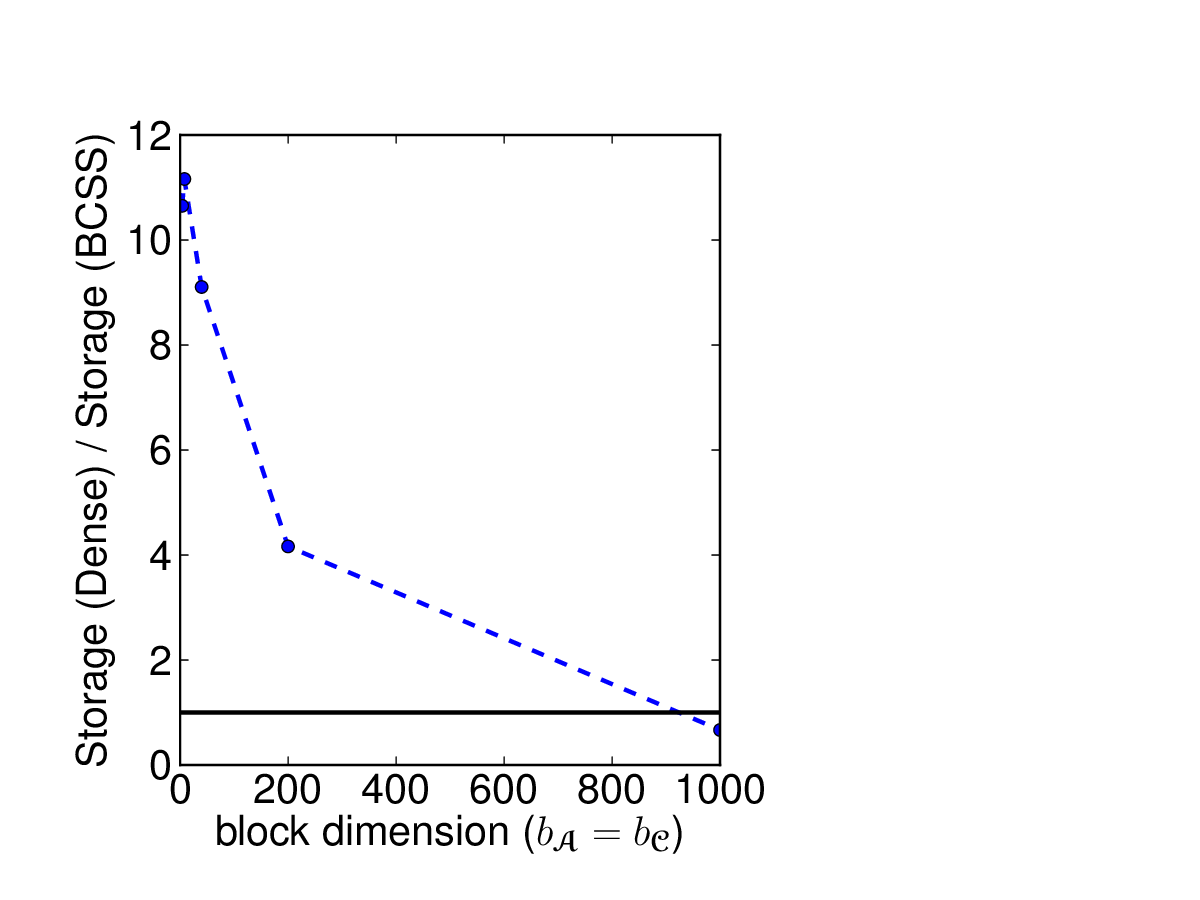}
\end{tabular}
\caption{Experimental results when the order $ m = 3 $, $ n = p = 1000 $ and 
the block dimensions $ b_{\T{A}} = b_{\T{C}} $ are varied.  The solid black line is used to indicate a unit ratio.}
\label{fig:varyBOrder3Tests}
\end{figure}

\begin{figure}[h!]
\centering
\includegraphics[width=3.0in]{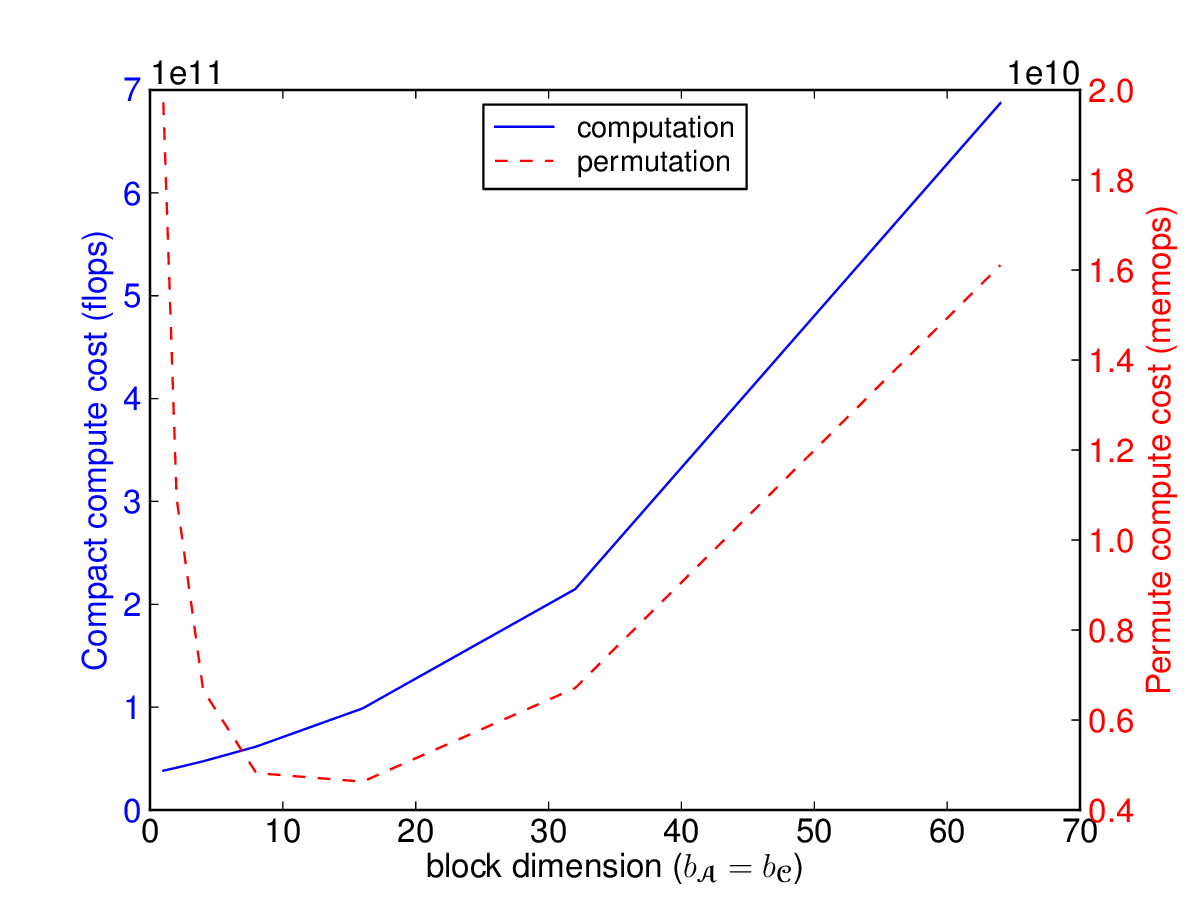}
\caption{Comparison of computational cost to permutation cost where $m=5$, $n = p = 64$ and tensor block dimension ($\blkDim{\T{A}}$, $\blkDim{\T{C}}$) is varied. The solid line represents the number of flops (due to computation) required for a given problem (left axis), and the dashed line represents the number of memops (due to permutation) required for a given problem (right axis).}
\label{fig:ComputePermuteCountComparison}
\end{figure}

\subsection{Results}

Figures~\ref{fig:varyMTests}--\ref{fig:varyBOrder3Tests} show results from
executing our implementation on the target architecture.  The \Dense algorithm
does not take advantage of symmetry nor blocking of the data objects, whereas
the \BCSSshort algorithm takes advantage of both.
All figures show comparisons of the execution time of each algorithm, the
associated speedup of the \BCSSshort algorithm over the \Dense algorithm, and the
estimated storage savings factor of the \BCSSshort algorithm not including
storage required for meta-data.

For the experiments reported in Figure~\ref{fig:varyMTests} we fix the
dimensions $ n $ and $ p $, and the block-sizes $ b_{\T{A}} $ and $ b_{\T{C}} $,
and vary the tensor order $ m $.
Based on experiment, the \BCSSshort algorithm begins to outperform the \Dense
algorithm after the tensor order is greater than or equal to
$4$.  This effect for small $m$ should be understood in context of the experiments performed.
As $n=p=16$, the problem size is quite small (the problem for $m=2$ and $m=3$
are equivalent to a matrices of size $16 \times 16$ and $64 \times 64$
respectively) reducing the benefit of storage-by-blocks (as storing the entire
matrix contiguously requires minor space overhead but benefits greatly from more
regular data access). Since such small problems do not provide
useful comparisons for the reader, the results of using the \BCSSshort{}
algorithm with problem parameters $m=3$, $n=p=1000$, and varied block-dimensions
are given in \Fig{varyBOrder3Tests}.  \Fig{varyBOrder3Tests} shows that the
\BCSSshort algorithm is able to outperform the \Dense algorithm given a large
enough problem size and an appropriate block size.  Additionally, notice that
\BCSSshort allows larger problems to be solved; the \Dense algorithm was unable
to compute the result when an order-$8$ tensor was given as input
due to an inability to store the problem in memory.

Our model predicts that the \BCSSshort algorithm should achieve an
$O\left((m+1)!/2^m\right)$ speedup over the \Dense algorithm.  Although it
appears that our experiments are only achieving a linear speedup over the \Dense
algorithm, this is because the values of $m$ are so small that the predicted
speedup factor is approximately linear with respect to $m$.  In terms of storage
savings, we would expect the \BCSSshort algorithm to have an $O\left(m!\right)$
reduction in space over the \Dense algorithm.  The fact that we are not seeing
this in the experiments is because the block dimensions $\blkDim{\T{A}}$ and
$\blkDim{\T{C}}$ are relatively large when compared to the tensor dimensions $n$
and $p$ meaning the \BCSSshort algorithm does not have as great an opportunity
to reduce storage requirements.

In Figure~\ref{fig:varyNTests} we fix the order $ m $, the block-sizes $
b_{\T{A}} $, and $ b_{\T{C}} $ and vary the tensor dimensions $ n $ and $ p $.
We see that the \BCSSshort algorithm outperforms the \Dense algorithm and attains
a noticeable speedup.
The experiments show a roughly linear speedup when viewed relative to $n$ with
perhaps a slight leveling off effect towards larger problem dimensions.  We
would expect the \BCSSshort algorithm to approach a maximum speedup relative to
the \Dense algorithm.  According to \Fig{dense_compact_comparison}, we would
expect the speedup of the \BCSSshort algorithm over the \Dense algorithm to level
off completely when our problem dimensions ($n$ and $p$) are on the order of
$400$ to $600$.  Unfortunately, due to space limitations we were unable to test
beyond the $n=p=64$ problem dimension and therefore were unable to completely
observe the leveling off effect in the speedup.

In Figure~\ref{fig:varyBTests}  we fix $ m $, $ n $, and $ p $, and vary the
block sizes $ b_{\T{A}} $ and $ b_{\T{C}} $.
The right-most point on the axis corresponds to the dense case (as
$\blkDim{\T{A}} = n = \blkDim{\T{C}} = p$) and the left-most point corresponds
to the fully-compact case (where only unique entries are stored).
There now is a range of block dimensions for which the \BCSSshort algorithm
outperforms the \Dense algorithm. Further, the \BCSSshort algorithm performs as
well or worse than the \Dense counterpart at the two endpoints in the graph.  This is
expected toward the right of the figure as the \BCSSshort algorithm reduces to
the \Dense algorithm, however the left of the figure requires a different
explanation.

In \Fig{ComputePermuteCountComparison}, we illustrate (with predicted flop and
memop counts) that there exists a point where smaller block dimensions
\emph{dramatically} increases the number of memops required to compute the
operation.  Although a smaller block dimension results in less flops required
for computing, the number of memops required increases significantly more.  As
memops are typically significantly more expensive than flops, we can expect that
picking too small a block dimension can be expected to drastically degrade
overall performance.

\section{Conclusion and Future Work}
We present storage by blocks, \BCSSshort, for tensors and show how this
can be used to compactly store symmetric tensors.
The benefits are demonstrated with an implementation of a new algorithm for the
change-of-basis ({\tt  sttsm}) operation.  Theoretical and practical results
show that both the storage and computational requirements are reduced
relative to storing the tensors densely and computing without taking advantage
of symmetry.

This initial study exposes many new research opportunities for extending
insights from the field of high-performance linear algebra to multi-linear
computation, which we believe to be the real contribution of this paper.
We finish by discussing some of these opportunities.

\paragraph*{Optimizing tensor permutations}

In our work, we made absolutely no attempt to optimize the tensor permutation
operation.  Without doubt, a careful study of how to organize these
tensor permutations will greatly benefit performance.  It is likely that the current
implementation not only causes unnecessary cache misses, but also a great number
of Translation Lookaside Buffer (TLB) misses~\cite{Goto:2008:AHP}, which cause the core to stall for
a hundred or more cycles.

\paragraph*{Optimized kernels/avoiding tensor permutations}

A better way to mitigate the tensor permutations is to avoid them as much as possible.
If $ n = p $, the {\tt sttsm} operation performs $ O( n^{m+1} ) $ operations on
$ O( n^m ) $ data.  This exposes plenty of opportunity to optimize this kernel
much like {\tt dgemm}, which performs $ O( n^3 ) $ computation on $ O( n^2 ) $
data, is optimized.  For other tensor operations, the ratio is even more
favorable.

We are developing a BLAS-like library, BLIS~\cite{FLAWN66}, that allows matrix
operations with matrices that have both a row and a column stride, as opposed to
the traditional column-major order supported by the BLAS.  This means that
computation with a planar slice in a tensor can be passed into the BLIS
matrix-matrix multiplication routine, avoiding the explicit permutations that
must now be performed before calling {\tt dgemm}.  How to rewrite the
computations with blocks in terms of BLIS, and studying the performance
benefits, is a future topic of research.

One can envision creating a BLAS-like library for blocked tensor operations. 
One alternative for this is to apply the techniques developed as part of the
PHiPAC~\cite{PHiPAC1}, TCE, SPIRAL~\cite{Puschel:05}, or ATLAS~\cite{ATLAS}
projects to the problem of how to optimize computations with blocks.  This
should be a simpler problem than optimizing the complete tensor contraction
problems that, for example, TCE targets now, since the sizes of the operands are
restricted.  The alternative is to create microkernels for tensor computations,
similar to the microkernels that BLIS defines and exploits for matrix
computations, and to use these to build a high-performance tensor library that
in turn can then be used for the computations with tensor blocks.

\paragraph*{Algorithmic variants for the {\tt sttsm} operation}

For matrices, there is a second algorithmic variant for computing
$ \M{C} \becomes \M{X} \M{A} \M{X}^T $.  Partition $ \M{A} $ by rows
and $ \M{X} $ by columns:
\[
\M{A} = 
\left( \begin{array}{c}
\Vrow{A}_0 \\
\vdots  \\
\Vrow{A}_{n-1}
\end{array}
\right)
\quad
\mbox{and}
\quad
\M{X} = 
\left( \begin{array}{c c c}
\V{X}_0 &\cdots& \V{X}_{n-1}
\end{array}
\right).
\]
Then 
\[
\M{C} = 
\M{X} \M{A} \M{X}^T 
=
\left( \begin{array}{c c c}
\V{X}_0 &\cdots& \V{X}_{n-1}
\end{array}
\right)
\left( \begin{array}{c}
\Vrow{A}_0 \\
\vdots  \\
\Vrow{A}_{n-1}
\end{array}
\right)
\M{X}^T
=
\V{X}_0 ( \Vrow{A}_0 \M{X}^T ) + \cdots
\V{X}_{n-1} ( \Vrow{A}_{n-1} \M{X}^T ).
\]
We suspect that this insight can be extended to the {\tt sttsm} 
operation, yielding a new set of algorithm-by-blocks that will have 
different storage and computational characteristics.

\paragraph*{Extending the FLAME methodology to multi-linear operations}

In this paper, we took an algorithm that was 
systematically derived with the FLAME methodology for the matrix case and
then extended it to the equivalent tensor computation.
Ideally, we would derive algorithms directly from the specification of
the tensor computation, using a similar methodology.  This requires a
careful consideration of how to extend the FLAME notation for
expressing matrix algorithms, as well as how to then use that notation
to systematically derive algorithms.

\paragraph*{Multithreaded parallel implementation}

Multithreaded parallelism can be accomplished in a number of ways.
\begin{itemize}
\item
The code can be linked to a multithreaded implementation of the BLAS, thus attaining parallelism within the {\tt dgemm} call.  This would require one to hand-parallelize the permutations.
\item
Parallelism can be achieved by scheduling the operations with blocks to threads much like the SuperMatrix~\cite{SuperMatrix:TOMS} runtime does for the {\tt libflame} library, or PLASMA~\cite{MAGMA} does for its tiled algorithms.
\end{itemize}
We did not yet pursue this because at the moment the permutations contribute a
significant overhead to the overall computation which we speculate consumes
significant bandwidth.  As a result, parallelization does not make sense until
the cost of the permutations is mitigated.

\paragraph*{Exploiting accelerators}

In a large number of papers~\cite{Igual20121134,LowHangingFruit,PDSEC:09,MAGMA}, we and others have shown how the
algorithm-by-blocks (tiled algorithm)
approach, when combined with a runtime system, can exploit
(multiple) GPUs and other accelerators.  These techniques can be
naturally extended to accommodate the algorithm-by-blocks for tensor computations.

\paragraph*{Distributed parallel implementation}

Once we understand how to derive sequential algorithms, it becomes
possible to consider distributed memory parallel implementation.  It
may be that our insights can be incorporated into the Cyclops Tensor
Framework~\cite{CTF}, or that we build on our own experience with 
distributed memory libraries for dense matrix computations, the
PLAPACK~\cite{PLAPACK} and Elemental~\cite{Poulson:2012:ENF} libraries, to develop
a new distributed memory tensor library.

\paragraph*{General multi-linear library}

The ultimate question is, of course, how the insights in this paper
and future ones can be extended to a general, high-performance
multi-linear library, for all platforms.


\subsection*{Acknowledgments}
We would like to thank Grey Ballard for his insights in restructuring many parts
of this paper.
This work was partially sponsored by NSF grants ACI-1148125 and CCF-1320112.
This work was also supported by the Applied Mathematics program at the U.S.
Department of Energy.
Sandia National Laboratories is a multi-program laboratory managed and operated
by Sandia Corporation, a wholly owned subsidiary of Lockheed Martin Corporation,
for the U.S. Department of Energy's National Nuclear Security Administration
under contract DE-AC04-94AL85000.

 {\em Any opinions, findings and conclusions or recommendations
 expressed in this material are those of the author(s) and do not
 necessarily reflect the views of the National Science Foundation
 (NSF).}

\bibliographystyle{plain}


\newpage

\appendix
\section{Casting Tensor-Matrix Multiplication to BLAS}
\label{appendix:permutations} 
Given a tensor $\T{A} \in \mathbb{R}^{I_\IdxFirst \times \cdots \times
I_{\IdxLast{m}}}$, a mode $k$, and a matrix $\M{B} \in \mathbb{R}^{J \times
I_k}$, the result of multiplying $\M{B}$ along the $k$-th mode of
$\T{A}$ is denoted by 
$
\T{C} = \T{A} \times_k \M{B}
$,
where $\T{C} \in \mathbb{R}^{I_\IdxFirst \times \cdots \times I_{k-1} \times J
\times I_{k+1} \times \cdots \times I_{\IdxLast{m}}}$ and each element of
$\T{C}$ is defined as
\[
\T{C}_{\subthree{\subfour{i_\IdxFirst}{\cdots}{i_{k-1}}}{j_\IdxFirst}{\subthree{i_{k+1}}{\cdots}{i_{\IdxLast{m}}}}} = \sum_{i_k=\IdxFirst}^{I_k} \alpha_{\subthree{i_\IdxFirst}{\cdots}{i_{\IdxLast{m}}}} \beta_{j_\IdxFirst i_k}
.
\]
This operation is typically computed by casting it as a matrix-matrix
multiplication for which high-performance implementations are
available as part of the Basic Linear Algebra Subprograms (BLAS)
routine {\tt dgemm}.

The problem viewing a higher-order tensor as a matrix is analogous to
the problem of viewing a matrix as a vector.  We first describe this simpler
problem and show how it generalizes to objects of higher-dimension.

\paragraph*{Matrices as vectors (and vice-versa).} 
A matrix $\M{A} \in \matsz{m}{n}$ can be viewed as a vector $\V{A} \in
\mathbb{R}^{M}$ where $M = mn$ by assigning $\V{A}_{i_\IdxFirst + i_\IdxSecond
m} = \M{A}_{\subtwo{i_\IdxFirst}{i_\IdxSecond}}$.  (This is analogous
to column-major order assignment of a matrix to memory.)
This alternative view does not change the relative order of the elements in the matrix,
since it just logically views them in a different way.  We say that the two dimensions of
$\M{A}$ are merged or ``grouped'' to form the single index of $\V{A}$.

Using the same approach, we can view $\V{A}$ as $\M{A}$ by assigning the
elements of $\M{A}$ according to the mentioned equivalence.
In this case, we are in effect viewing the single index of $\V{A}$ as two
separate indices.  We refer to this effect as a ``splitting'' of the index of
$\V{A}$.

\paragraph*{Tensors as matrices (and vice-versa).} 
A straightforward extension of grouping of indices allows us to view
higher-order tensors as matrices and (inversely)
matrices as higher-order tensors.  The difference lies with the calculation used
to assign elements of the lower/higher-order tensor.

As an example, consider an order-4 tensor $\T{C} \in \mathbb{R}^{I_\IdxFirst
\times I_\IdxSecond \times I_\IdxThird \times I_{\IdxFourth}}$.  We can view
$\T{C}$ as a matrix $\M{C} \in \mathbb{R}^{J_\IdxFirst \times J_\IdxSecond}$
where $J_\IdxFirst = I_\IdxFirst \times I_\IdxSecond$ and $J_1 = I_\IdxThird
\times I_\IdxFourth$.  Because of this particular grouping of indices, the
elements as laid out in memory need not be rearranged (relative order of each
element remains the same).   This follows from the observation 
that memory itself is a linear array (vector) and realizing that 
if $\M{C}$
and $\T{C}$ are both mapped to a 1-dimensional vector using
column-major
order and its higher dimensional extension (which we call
dimensional order), 
both are stored identically.

\paragraph*{The need for permutation.}
If we wished to instead view our example \newline $\T{C} \in \mathbb{R}^{I_\IdxFirst \times
I_\IdxSecond \times I_\IdxThird \times I_{\IdxFourth}}$ as a matrix $\M{C} \in
\mathbb{R}^{J_\IdxFirst \times J_\IdxSecond}$ where, for instance, $J_\IdxFirst =
I_\IdxSecond$ and $J_\IdxSecond = I_\IdxFirst \times I_\IdxThird \times
I_\IdxFourth$, then this would require a rearrangement of the data
since mapping
$\M{C}$ and $\T{C}$ to memory using dimensional order will not
generally produce the same result for both.
This is a consequence of changing the relative order of indices in our
mappings.

This rearrangement of data is what is referred to as a \emph{permutation} of
data.  By specifying an input tensor $\T{A} \in \mathbb{R}^{I_\IdxFirst \times
\cdots \times
  I_{\IdxLast{m}}}$ and the desired permutation of indices of $\T{A}$,
$\pi$, we define the transformation $ \T{C} = {\rm permute}(\T{A}, \pi) $ that
yields $\T{C} \in \mathbb{R}^{I_{\pi_{\IdxFirst}} \times I_{\pi_{\IdxSecond}}
\times \cdots \times I_{\pi_{\IdxLast{m}}}}$ so that
$\T{C}_{\subthree{i'_\IdxFirst}{\cdots}{i'_{\IdxLast{m}}}} =
\T{A}_{\subthree{i_{\IdxFirst}}{\cdots}{i_{\IdxLast{m}}}}$ where $i'$
corresponds to the result of applying the permutation $\pi$ to $i$.
The related operation ${\rm ipermute}$ inverts this transfomation when supplied
$\pi$ so that $ \T{C} = ipermute(\T{A}, \pi) $ yields $\T{C} \in
\mathbb{R}^{I_{\pi^{-1}_{\IdxFirst}} \times I_{\pi^{-1}_{\IdxSecond}} \times
\cdots \times I_{\pi^{-1}_{\IdxLast{m}}}}$ where
$\T{C}_{\subthree{i'_\IdxFirst}{\cdots}{i'_{\IdxLast{m}}}} =
\T{A}_{\subthree{i_{\IdxFirst}}{\cdots}{i_{\IdxLast{m}}}}$ where $i'$
corresponds to the result of applying the permutation $\pi^{-1}$ to $i$.

\paragraph*{Casting a tensor computation in terms of a matrix-matrix multiplication.}

We can now show how the operation
$
\T{C} = \T{A} \times_k \M{B}
$,
where $\T{A} \in \mathbb{R}^{I_\IdxFirst \times \cdots \times I_{\IdxLast{m}}}$,
$\M{B} \in \mathbb{R}^{J \times I_k}$, and $\T{C} \in \mathbb{R}^{I_\IdxFirst
\times \cdots \times I_{k-1} \times J \times I_{k+1} \times \cdots \times
I_{\IdxLast{m}}}$, can be cast as a matrix-matrix multiplication if
the tensors are appropriately permuted. The following describes the algorithm:
\begin{enumerate}
  \item Permute: $\T{P}_{\T{A}} \leftarrow {\rm permute}(\T{A}, \{k,\IdxFirst,\ldots,k-1,k+1,\ldots,\IdxLast{m}\})$.
  \item Permute: $\T{P}_{\T{C}} \leftarrow {\rm permute}(\T{C}, \{k,\IdxFirst,\ldots,k-1,k+1,\ldots,\IdxLast{m}\})$.
  \item View tensor $\T{P}_{\T{A}} $ as matrix $\M{A} $: $\M{A} \leftarrow \T{P}_{\T{A}}$, where $\M{A} \in \mathbb{R}^{I_k \times J_\IdxSecond}$ and $J_\IdxSecond = I_\IdxFirst \cdots I_{k-1} I_{k+1} \cdots I_{\IdxLast{m}}$.
  \item View tensor $\T{P}_{\T{C}} $ as matrix $\M{C} $: $\M{C} \leftarrow \T{P}_{\T{C}}$, where $\M{C} \in \mathbb{R}^{J \times J_\IdxSecond}$ and $J_\IdxSecond = I_\IdxFirst \cdots I_{k-1} I_{k+1} \cdots I_{\IdxLast{m}}$.
  \item Compute matrix-matrix product: $\M{C} \becomes \M{B}\M{A}$.
  \item View matrix $\M{C} $ as tensor $\T{P}_{\T{C}} $: 
$\T{P}_{\T{C}} \leftarrow \M{C}$, \newline where $\T{P}_{\T{C}} \in \mathbb{R}^{J \times I_{\IdxFirst} \times \cdots \times I_{k-1} \times I_{k+1} \times \cdots \times I_{\IdxLast{m}}}$.
  \item ``Unpermute'': $\T{C} \leftarrow {\rm ipermute}(\T{P}_{\T{C}}, \{k, \IdxFirst, \ldots, k-1,k+1,\ldots,\IdxLast{m}\})$.
\end{enumerate}
Step 5. can be implemented by a call to the BLAS routine {\tt dgemm},
which is typically highly optimized.

\section{Design Details}
\label{appendix:meta-data}
\begin{figure}[tbp]
\centering
\begin{tabular}{@{} c @{\hspace{-0.85in}} c}
\includegraphics[width=3.45in]{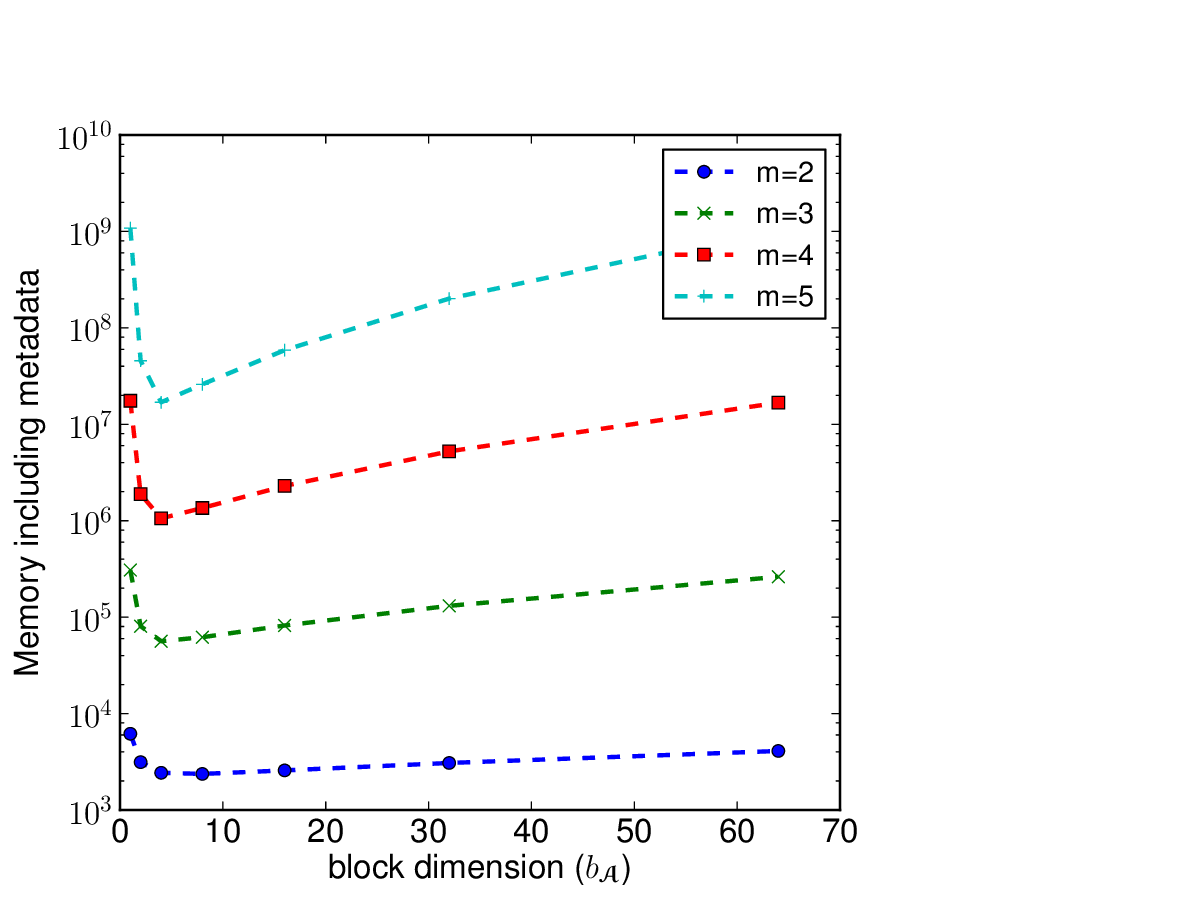} &
\includegraphics[width=3.45in]{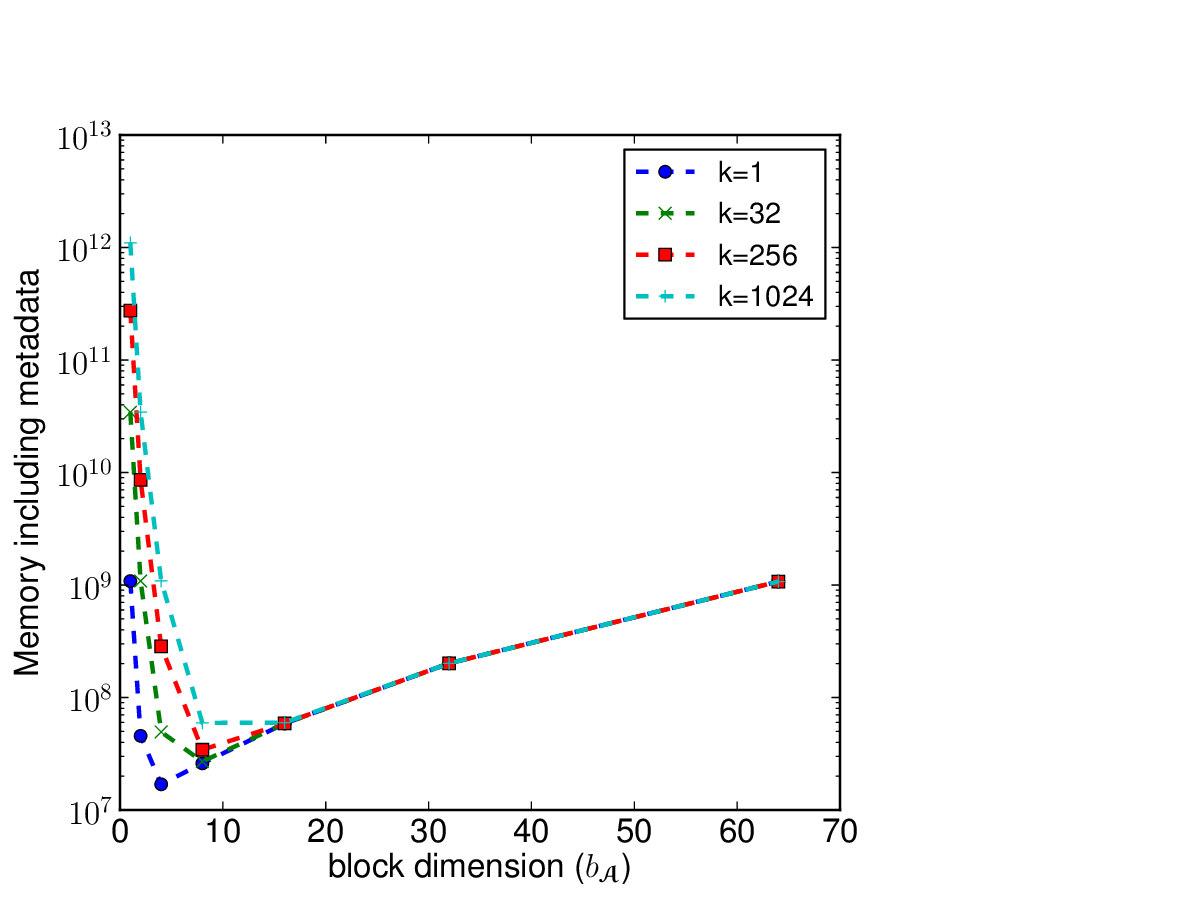} \\
\end{tabular}
\caption{Left: Storage requirements of $\T{A} \in \tensz{m}{64}$ as block dimension 
changes. Right: Storage requirements of $\T{A} \in \tensz{5}{64}$ for
different choices for the meta-data  stored per block, measured by $k$, as block dimension changes.}
\label{fig:varyBActualStorage}
\end{figure}

We now give a few details about the particular implementation of
\BCSSshort,
and how this impacts storage requirements.  Notice that this is one
choice for implementing this storage scheme in practice.  One can
envision other options that, at the expense of added complexity in the
code, reduce the memory footprint.

\BCSSshort views tensors hierarchically.
At the top level, there is a tensor where each element of that tensor is itself
a tensor (block).
Our way of implementing this stores a description (meta-data) for a block in
each element of the top-level tensor. 
This meta-data adds to memory requirements.
In our current implementation, the top-level tensor of meta-data is itself a
dense tensor.  The meta-data in the upper hypertriangular tensor describes
stored blocks.  The meta-data in the rest of the top-level tensor reference the
blocks that correspond to those in the upper hypertriangular tensor (thus
requiring an awareness of the permutation needed to take a stored block and
transform it).
This design choice greatly simplifies our implementation (which we hope to
describe in a future paper).
We show that although the meta-data can potentially require considerable space,
this can be easily mitigated.  We use $\T{A}$ for example purposes.

Given $\T{A} \in \tensz{m}{n}$ stored with \BCSSshort with block dimension
$\blkDim{\T{A}}$, we must store meta-data for $\blkedDimA^m$ blocks where
$\blkedDimA = \lceil n / \blkDim{\T{A}} \rceil$.  This means that the total cost
of storing $\T{A}$ with \BCSSshort is
\[
C_{\rm storage}(\T{A}) = k\blkedDimA^m +
\blkDim{\T{A}}^m\binom{\blkedDimA + m - 1}{m}
\mbox{~floats},
\]
$k$ is a
constant representing the amount of storage required for the meta-data
associated with one block, in floats.  Obviously, this meta-data is
of a different datatype, but floats will be our unit.

There is a tradeoff between the cost for storing the
meta-data
and the actually entries of $ \T{A} $, parameterized by 
the blocksize $\blkDim{\T{A}} $:
\begin{itemize}
\item
If $\blkDim{\T{A}}=n$,
then we only require a minimal amount of memory for meta-data, $ k $
floats, but must store all entries of $\T{A}$ since there now is only
one block, and that block uses dense storage.  We thus store slightly
more than we would if we stored the tensor without symmetry.
\item
If 
$\blkDim{\T{A}}=1$, 
then $\blkedDimA=n$ and we must store meta-data for each
element, meaning we store much more data than if we just used a dense
storage scheme.  
\end{itemize}
Picking a block dimension somewhere between these two extremes results in
a smaller footprint overall.  For example, if we choose a block dimension
$\blkDim{\T{A}}=\sqrt{n}$, then $\blkedDimA=\sqrt{n}$ and the total storage
required to store $\T{A}$ with \BCSSshort is
\[
\begin{array}{l c l}
C_{\rm storage(\T{A})} & = & k\blkedDimA^{m} + \blkDim{\T{A}}^m\binom{\blkedDimA + m - 1}{m} 
 =  kn^{\frac{m}{2}} + n^{\frac{m}{2}}\binom{n^{\frac{m}{2}} + m - 1}{m} \\
 & \approx & kn^{\frac{m}{2}} + n^{\frac{m}{2}}\left(\frac{n^{\frac{m}{2}}}{m!}\right) 
 =  n^{\frac{m}{2}}\left(k + \frac{n^{\frac{m}{2}}}{m!}\right)
 \mbox{~\rm floats},\\
 \end{array}
 \]
 which, provided that $k \ll \frac{n^{\frac{m}{2}}}{2}$, is significantly smaller
 than the storage required for the dense case ($n^m$).  
This discussion suggests that a point exists that
 requires less storage than the dense case (showing that \BCSSshort
 is a feasible solution).
 
In Figure~\ref{fig:varyBActualStorage}, we illustrate that as long as we pick a
block dimension that is greater than $4$, we avoid incurring extra costs due to
meta-data storage.  It should be noted that changing the dimension of the
tensors also has no effect on the minimum, however if they are too small, then
the dense storage scheme may be the minimal storage scheme.
Additionally, adjusting the order of tensors has no real effect on the block
dimension associated with minimal storage. However increasing the amount of
storage allotted for meta-data slowly shifts the optimal block dimension choice
towards the dense storage case.

\end{document}